\documentclass[11pt,oneside]{amsart}
\usepackage[margin=1.05in]{geometry}
\usepackage[T1]{fontenc}
\usepackage{lmodern,amsmath,amssymb,amsthm,mathtools,mathrsfs,booktabs,array}
\usepackage{microtype}
\usepackage{xcolor}
\usepackage{comment}
\usepackage[colorlinks=true,linkcolor=blue!55!black,citecolor=blue!55!black,urlcolor=blue!55!black]{hyperref}
\usepackage[nameinlink,noabbrev]{cleveref}
\hypersetup{
 pdftitle={Sharp Constants and Extremizers in Homogeneous Tree BMO},
 pdfauthor={Mohammadreza Mohajer},
 pdfkeywords={tree BMO, absolute value, finite-tree reduction, stable centering, Jensen defects}
}
\newtheorem{thm}{Theorem}[section]
\newtheorem{prop}{Proposition}[section]
\newtheorem{lemma}{Lemma}[section]
\newtheorem{cor}{Corollary}[section]
\theoremstyle{definition}
\newtheorem{defn}{Definition}[section]

\theoremstyle{remark}
\newtheorem{remark}{Remark}[section]
\crefname{thm}{Theorem}{Theorems}
\Crefname{thm}{Theorem}{Theorems}
\crefname{prop}{Proposition}{Propositions}
\Crefname{prop}{Proposition}{Propositions}
\crefname{lemma}{Lemma}{Lemmas}
\Crefname{lemma}{Lemma}{Lemmas}
\crefname{cor}{Corollary}{Corollaries}
\Crefname{cor}{Corollary}{Corollaries}
\crefname{defn}{Definition}{Definitions}
\Crefname{defn}{Definition}{Definitions}
\newcommand{\R}{\mathbb R}

\newcommand{\Z}{\mathbb Z}

\newcommand{\E}{\mathbb E}
\newcommand{\Prob}{\mathbb P}
\newcommand{\cT}{\mathcal T}

\newcommand{\cA}{\mathcal A}
\newcommand{\scN}{\mathscr N}
\newcommand{\scC}{\mathscr C}
\newcommand{\scA}{\mathscr A}
\newcommand{\one}{\mathbf 1}
\DeclareMathOperator{\BMO}{BMO}
\DeclareMathOperator{\osc}{osc}
\DeclareMathOperator{\Carl}{Carl}
\DeclareMathOperator{\Var}{Var}
\DeclareMathOperator{\sgn}{sgn}

\newcommand{\avg}[1]{\langle #1\rangle}
\newcommand{\ie}{{\it i.e., }}
\numberwithin{equation}{section}

\title{Sharp Constants and Extremizers in Homogeneous Tree BMO}

\author{Mohammadreza Mohajer}
\address{Department of Mathematics, Physics and Geology, Cape Breton University, 1250 Grand Lake Road, Sydney, NS B1P 6L2 Canada}
\email{mohammadreza.mohajer96@gmail.com}

\thanks{}

\keywords{Tree BMO, homogeneous trees, finite-tree reduction, harmonic analysis on local fields,\newline absolute value, rational extremizers, Jensen defects}
\subjclass[2020]{Primary 42B35; Secondary 26D15, 60G46, 11S80}
\begin{document}
\begin{abstract}
We develop a finite-tree formalism for sharp inequalities on real-valued,
mean-centered $\operatorname{L}^1$-BMO over homogeneous $q$-ary
filtrations. A reduction theorem identifies local and global sharp
constants with the supremum of finite-depth constants. Finite anchoring
gives compact optimization problems and rational extremizers under
rational piecewise-linear hypotheses, while stable centering shows
that imposing zero root average preserves the supremum over all depths.
The principal application is the absolute-value map: we construct
centered examples in which taking absolute values increases the BMO
seminorm and obtain bounds strictly below the universal factor $2$.
Jensen-defect Carleson functionals and martingale-difference envelopes
provide further applications, illustrating distinct anchoring
mechanisms and the differences between finite-depth rationality,
stabilization, and global attainment.
\end{abstract}
\maketitle
\enlargethispage{6pt}
\setcounter{tocdepth}{1}
\section{Introduction}\label{sec:intro}

Sharp inequalities in BMO depend on both the family of averaging
sets and the choice of oscillation seminorm. The absolute-value
map provides a basic example. For the mean-centered
$\operatorname{L}^1$-BMO seminorm introduced by John and Nirenberg
\cite{JohnNirenberg1961}, the elementary estimate
\[
 \|\,|f|\,\|_\ast\le2\|f\|_\ast
\]
holds for any family of averaging sets. We call the absolute-value map contractive if
$\|\,|f|\,\|_\ast\le\|f\|_\ast$ for every $f$.
We say that amplification occurs when
$\|\,|f|\,\|_\ast>\|f\|_\ast$ for some $f$. Korenovskii proved that
the optimal constant is $1$ when all subintervals of an interval
are used \cite{Korenovskii1992}. The corresponding problem for
more general averaging families is discussed by Dafni and Gibara
\cite[Section~5]{Dafni2019BMOonShapes}. On a fixed dyadic
filtration, the arrangement of values across generations creates
different extremal behaviour.

We develop a finite-tree framework for studying sharp constants
of functionals on BMO. The framework connects inequalities on
an entire filtration, inequalities on a rooted descendant tree,
and finite-dimensional optimization over leaf values. Its main
contributions concern finite reduction, attainment, rationality,
and the effect of imposing zero root average. Three applications
illustrate these results: Jensen-defect Carleson functionals,
the absolute-value map, and martingale-difference envelopes.

Our setting is a homogeneous $q$-ary filtration $\cT$ on a
measure space $(X,\mu)$, in which every node has $q$ children
of equal measure. Slavin and Vasyunin study rooted
$\alpha$-trees, allowing unequal child measures bounded below
by the fraction $\alpha$ of their parent measure
\cite[Definition~1.4 and Remark~1.5]{SlavinVasyunin2016AlphaTrees}.
Each rooted descendant tree in our setting is a homogeneous
special case with $\alpha=1/q$. Our global filtration has
generations indexed by $\Z$, so every node has ancestors
at arbitrarily large scales (\cref{def:tree}). This
bi-infinite structure accommodates both standard dyadic
cubes in $\R^d$, with $q=2^d$, and balls in non-Archimedean
local fields with residue cardinality prime power $q$. For BMO and
singular integrals over local fields, see
\cite{DalyKurtz1993BMOlocalfields}.

For real-valued functions, we use
\[
 f_I:=\frac1{\mu(I)}\int_I f\,d\mu,
 \qquad
 \osc_I(f):=\frac1{\mu(I)}\int_I|f-f_I|\,d\mu,
 \qquad
 \|f\|_\ast:=\sup_{I\in\cT}\osc_I(f).
\]
The local seminorm $\|f\|_{\ast,B}$ takes the supremum over
descendants of a fixed root $B$. The choice of
$\operatorname{L}^1$-BMO is important here. Absolute value
does not increase the mean-centered $\operatorname{L}^2$-BMO
seminorm, since taking absolute values cannot increase
variance. For $\operatorname{L}^1$-BMO, the arithmetic mean
need not minimize average absolute deviation, and the
composition constant can exceed $1$. The extension,
finite reduction, anchoring, and centering methods developed here extend
to $\operatorname{L}^p$-BMO, $1\le p<\infty$, with suitable
bounds. The $\operatorname{L}^1$ setting has the further
advantage that its finite-depth constraints are given by
rational linear inequalities for which the associated finite-depth constants are rational with rational maximizer.

The general framework is developed in \Cref{sec:formalism}.
For compatible positively homogeneous local functionals
satisfying a lower semicontinuity condition, the local and
global sharp constants equal the supremum of the
finite-depth constants (\cref{thm:reduction}). Finite
anchoring accounts for the dependence of a functional
on additive constants and reduces the finite-depth
problem to compact sets. When the functional is rational
piecewise linear and the anchors are rational, the sharp
constant is rational and is attained by a rational vector
(\cref{thm:finite-anchor}).

Stable centering relates the centered and uncentered
problems. For hereditary functionals admitting zero-leaf
anchoring, every admissible depth-$m$ configuration gives
a centered admissible configuration of depth at most
$m+\lceil qm/(q-1)\rceil$ with no smaller functional value
(\cref{thm:stable-centering}). Consequently, the centered
and uncentered constants have the same supremum over all
depths. A corresponding result holds for general finite
anchors (\cref{cor:general-centering}).

The first application concerns Jensen defects
(\Cref{sec:jensen}). The Jensen gap on a node decomposes
exactly into nonnegative contributions from successive
generations, and the Carleson norm of these contributions
equals the supremum of the Jensen gaps. 
For $\Phi(t)=|t|$, optimizing over translations recovers
the input BMO seminorm and gives node anchoring, while
anchoring at leaf values fails (\cref{thm:node-anchor}).
For every convex Lipschitz $\Phi$, the sharp packing
constant is half the difference between its limiting
slopes at $+\infty$ and $-\infty$, whether or not $\Phi$
is homogeneous (\cref{thm:slope-variation}).

The principal application is the sharp constant $c(q)$ on a bi-infinite $q$-ary tree in
\[
 \|\,|f|\,\|_\ast\le c(q)\|f\|_\ast.
\]
Write $c_m(q)$ for the corresponding constant at depth
$m$ and $\bar c_m(q)$ for the constant with zero root
average (centered case). The absolute-value functional admits zero-leaf
anchoring, and the formalism gives
\begin{equation}\label{eq:intro-centering}
 c(q)=\sup_{m\ge1}c_m(q)=\sup_{m\ge1}\bar c_m(q).
\end{equation}
Both finite-depth constants are rational and attained
by rational vectors (\cref{prop:absolute-structure}),
and an explicit family of rational linear programs
computes $c_m(q)$ (\Cref{exact-family-finite-linear-programs}).

A sharp depth-$1$ constant for every $q>1$ is
\begin{equation}\label{eq:intro-shallow}
 c_1(q)=\max\{1,2-4/q\}
\end{equation}
(\cref{thm:scalar-sharp,cor:finite-abs}).
Contractivity fails at greater binary depths
(\cref{thm:binary-counterexample}), where binary refers to the case $q=2$. A recursive binary
construction gives the lower bound $25/24$, while a
centered depth-three construction for $q\ge3$ improves
the depth-two bound $\Delta_q:=2-4/q+4/q^2$
(\cref{thm:binary-recursive,thm:q-counterexample}).
Together with the uniform upper estimate in
\cref{thm:global-upper}, these results yield
\begin{align}
 \frac{25}{24}
 &\le c(2)\le\frac{18}{11},
 \label{eq:intro-binary}\\
 \Delta_q+\frac{q-2}{q^3}
 &\le c(q)\le\frac{2(q+1)^2}{q^2+3q+1},
 \qquad q\ge3.
 \label{eq:intro-general}
\end{align}
The lower constructions can be centered, so requiring
zero root average does not restore contractivity.

The upper estimate is strictly smaller than $2$ for
every fixed $q$. It combines a scalar variance bound
with Slavin and Vasyunin's comparison between
$\operatorname{L}^1$- and $\operatorname{L}^2$-BMO
\cite[Theorem~1.8]{SlavinVasyunin2016AlphaTrees}.
This connects our approach with Bellman-function
methods for sharp BMO inequalities
\cite{SlavinVasyunin2016AlphaTrees,IvanisviliEtAl2018BMO2Evolution}.
The exact value of $c(q)$ remains open.

The final application studies envelopes of differences
between node averages and descendant averages
(\Cref{sec:martingale}). We determine the sharp constants
for the $\operatorname{L}^p$ norm of one-generation
differences and the sharp endpoint constants for
cumulative differences (\cref{thm:gradient,thm:cumulative}).
These examples clarify the role of the rationality
hypotheses: for $q=3$ and $p=2$, the one-generation
constant is already irrational at finite depth.
Weighted envelopes with rational coefficients have
rational constants at every finite depth, yet those
constants may increase strictly towards a rational
or irrational limit (\cref{prop:weighted}).
The global sharp constant can nevertheless be attained
on the bi-infinite filtration. These examples distinguish
finite-depth attainment, stabilization, and global
attainment within the same framework.

\section{Formalism over a Measure Space}\label{sec:formalism}
The viewpoint of oscillation relative to a family of measurable sets is
consistent with \cite{Dafni2019BMOonShapes}. The filtration below is an exactly homogeneous version of a rooted $\alpha$-tree
\cite[Definition~1.4 and Remark~1.5]{SlavinVasyunin2016AlphaTrees}, with ancestors added at every generation. Homogeneous trees occur in the study of dyadic-like maximal operators; see \cite{MelasDyadicLlogL}.

\begin{defn}[Strong homogeneous $\alpha$-tree]\label{def:tree}
Let $(X,\mu)$ be a measure space, and let $0<\alpha<1$. A \emph{strong homogeneous $\alpha$-tree} is a family
$\cT=\bigcup_{n\in\Z}\cT_n$ satisfying the following conditions.
\begin{enumerate}
\item For every $n\in\mathbb Z$, the family $\cT_n$ is a partition of $X$ up to null sets; \textit{i.e.}, \(X=\bigcup_{I\in\cT_n}I\) and that there exists a measurable
set \(X_n\subseteq X\), with \(\mu(X\setminus X_n)=0\), such
that every \(x\in X_n\) belongs to exactly one member of
\(\cT_n\).
\item Every $I\in\cT$ has finite positive measure.
\item The partition $\cT_{n+1}$ refines $\cT_n$. If $I\in\cT_n$, define its collection of children by \[\cT_1(I)
    :=
    \bigl\{
        J\in\cT_{n+1}: J\subseteq I
    \bigr\}.\] Then
    $I
    =
    \bigcup_{J\in\cT_1(I)}J$.
\item There is a fixed integer $q\ge2$ such that each node $I$ has exactly $q$
children, each of measure $\alpha\mu(I)$.
\item The filtration differentiates the space of \emph{tree-integrable functions}
\[
 \operatorname{L^1}(X,\cT):=\{f:X\to\R\text{ measurable}:\ \int_I|f|\,d\mu<\infty
                   \text{ for every }I\in\cT\}.
\]
The differentiation requirement is stated precisely on $X_\cT\subseteq X$ defined below.
\end{enumerate}
\end{defn}
The full-measure set \(X_n\) in item~(1) ensures that every
point outside a single measurable null set belongs to exactly
one node of \(\cT_n\). For an uncountable family, being pairwise disjoint up to null sets does not guarantee this property,
since their union need not be null. When \(\cT_n\) is countable,
the required set \(X_n\) is obtained by removing the union
of all pairwise intersections. Set
\begin{equation}\label{eq:common-domain}
 X_\cT:=\bigcap_{n\in\Z}X_n.
\end{equation}
Every $X_n$, and hence $X_\cT$, is measurable and has full measure. For $x\in X_\cT$, there is a unique $I_n(x)\in\cT_n$ containing $x$ at every generation, and $I_{n+1}(x)$ is a child of $I_n(x)$. The differentiation requirement in \cref{def:tree} is precisely
\[
 f_{I_n(x)}\longrightarrow f(x)\quad(n\to\infty)
 \qquad\text{for almost every }x\in X_\cT,
 \quad f\in \operatorname{L^1}(X,\cT).
\]
For the closed balls and half-open dyadic cubes in Subsections~~\ref{subsec:fields} and \ref{subsec:dyadic structure}, each $\cT_n$ is pairwise-disjoint, and so $X_\cT=X$.

Necessarily $\alpha=1/q$. 
The adjective \emph{strong} refers to the bi-infinite
generation structure, meaning each node $I\in\cT$ has an ancestor. This is in contrast to a rooted tree where $X\in\cT$. Each node $I$ has a unique generation $n$, called its \emph{depth}, and a
unique parent. Nodes that intersect on a set of positive measure are nested. If $I\in\cT_n$, write
\[
 \cT_k(I):=\{J\in\cT_{n+k}:J\subseteq I\},\qquad
 \cT(I):=\bigcup_{k\ge0}\cT_k(I).
\]
Then $|\cT_k(I)|=q^k$ and $\mu(J)=q^{-k}\mu(I)$ for
$J\in\cT_k(I)$. The descendant family $\cT(I)$ is a \emph{rooted} homogeneous
$\alpha$-tree. In particular, its space of tree-integrable functions is $\operatorname{L^1}(I)$. However, for a strong homogeneous $\alpha$-tree, $X$ cannot be a node, \ie $X\not\in\cT$, because $X$ has infinite measure and so a tree-integrable function on $\cT$ need not be in $\operatorname{L^1}(X)$.


Define $\BMO_\cT(X):=\{f\in \operatorname{L^1}(X,\cT):\|f\|_\ast<\infty\}$, 
\[
 \|f\|_{\ast,I}:=\sup_{J\in\cT(I)}\osc_J(f),\qquad
 \BMO_\cT(I):=\{f\in \operatorname{L^1}(I):\|f\|_{\ast,I}<\infty\}.
\]
The local seminorm vanishes exactly on functions constant almost everywhere
on $I$. Globally, its kernel consists of functions constant on each ancestor
component, rather than necessarily one constant on all of $X$. Here an ancestor component is the union of the successive
ancestors of a node.

\subsection{The setting over ultrametric local fields and dyadic
\texorpdfstring{$\R^d$}{Rd}}\label{subsec:setting}
Assume temporarily that every pair of points has a common ancestor:
\begin{equation}\label{eq:CA}
 \text{for every }x,y\in X,\text{ there is }n\in\Z
 \text{ with }I_n(x)=I_n(y).
\end{equation}
For $x,y\in X_\cT$, set
\[
 \ell(x,y):=\sup\{n\in\Z:I_n(x)=I_n(y)\},\qquad
 d_\cT(x,y):=q^{-\ell(x,y)},\quad q^{-\infty}:=0.
\]
\begin{prop}[The tree pseudo-ultrametric]\label{prop:metric}
Under \eqref{eq:CA}, $d_\cT$ is a pseudo-ultrametric on $X_\cT$ and
\[
 I_n(x)\cap X_\cT=\overline B_{d_\cT}(x,q^{-n})
       =B_{d_\cT}(x,q^{-(n-1)}), \quad \text{ for every }x\in X_\cT \text{ and }n\in\Z.
\]
Moreover, every positive-radius ball is the intersection of a node with $X_\cT$. If the partitions separate points, in the snse that for every $x,y\in X_\cT,\,x\neq y$ there exists $n\in\Z$ such that $I_n(x)\neq I_n(y)$, then
$d_\cT$ is an ultrametric.
\end{prop}
\begin{proof}
All points considered in this proof belong to \(X_\cT\). The common-ancestor assumption guarantees that set of common generations is nonempty. If $I_n(x)=I_n(y)$, then $I_m(x)=I_m(y)$ for every $m\leq n$, by the nesting of the generations. Consequently, $\ell(x,y)>-\infty$. If $x,y$
agree at generation $n$ and $y,z$ agree there, then $x,z$ agree there. Hence
$\ell(x,z)\ge\min\{\ell(x,y),\ell(y,z)\}$, which gives the strong triangle
inequality after applying $q^{-t}$. Symmetry and a zero diagonal are immediate.
Moreover, $y\in I_n(x)$ if and only if $\ell(x,y)\ge n$. Since the finite
values of $\ell$ are integers, this is equivalent both to
$d_\cT(x,y)\le q^{-n}$ and to $d_\cT(x,y)<q^{-(n-1)}$.
The same discreteness identifies balls of arbitrary positive radius.
Finally, zero distance means agreement at every generation, which is equality
of points under the separation assumption.
\end{proof}

\subsubsection{Local ultrametric fields}\label{subsec:fields}
Let $K$ be a non-Archimedean local field, \textit{i.e.}, a complete discretely valued field with respect to a non-Archimedean absolute value $|\cdot|_K$, which is locally compact. We denote by $\mathcal O_K:=\{x\in K:|x|_K\le 1\}$ the ring of integers of $K$. This is a local ring having residue field of size $q$. Choose a uniformizer \(t\)
and normalize the absolute value by \(|t|_K=q^{-1}\).
The nonzero elements have the unique form $ut^n$, with $u$ a unit in $\mathcal O_K^\times$
and $n\in\Z$. 
The field $K$ carries a \emph{doubling} Haar measure $\mu$, unique up to a multiplicative constant, normalized so that $\mu(\mathcal O_K)=1$. The closed balls are
\[
 B=a+t^n\mathcal O_K,\qquad \mu(B)=q^{-n}.
\]
Their elementary structure is described in \cite{SerreLocalFields,Taibleson1975}.
Every ball has $q$ children, the balls at each scale partition $K$, and any two
points have a common ancestor. Balls separate points. On each fixed ball $B$, the descendant balls form a differentiation basis.
Consequently these balls form a strong
homogeneous $q^{-1}$-tree. The associated space is the usual ball BMO on $K$;
see \cite{DalyKurtz1993BMOlocalfields} for its harmonic-analytic setting. Here $\operatorname{L^1}(\cT)=\operatorname{L^1}_{\mathrm{loc}}(K)$. Indeed, balls are compact and every
compact set is covered by finitely many balls. The kernel of $\|\cdot\|_\ast$ is
$\R$, and $\BMO(K)/\R$ is complete. 

\subsubsection{Dyadic \texorpdfstring{$\R^d$}{Rd}}\label{subsec:dyadic structure}
The standard half-open grid
\[
 \mathscr D(\R^d)=\{2^{-k}([0,1)^d+j):k\in\Z,\ j\in\Z^d\}
\]
has $q=2^d$ children per cube, each of relative measure $2^{-d}$.
It differentiates locally integrable functions and gives classical dyadic BMO.
Here tree integrability agrees with local integrability because a compact set
meets finitely many cubes at any fixed scale. The standard grid need not
satisfy \eqref{eq:CA}; already in one dimension, points on opposite sides of
zero have no common dyadic ancestor. Thus \cref{prop:metric} is not applied on dyadic case. For relations between dyadic and ordinary BMO, see \cite{GarnettJones1982}.

\subsection{A finite-tree reduction to step functions}
\label{subsec:reduction}
We label the children of each \(I \in \cT\) as
$\cT_1(I) = \{I_1, \dots, I_{q}\}$.
Then, for a word
$\omega = (\omega_1, \dots, \omega_k)
\in \{1, \dots, q\}^k$,
write \(I_\omega\) for the corresponding descendant of \(I\), defined recursively by
$I_{\omega_1,\dots,\omega_k}
=
\bigl(I_{\omega_1,\dots,\omega_{k-1}}\bigr)_{\omega_k}$.
Thus,
$\cT_k(I)
=
\{I_\omega : |\omega| = k\}$.
The empty word corresponds to the root.

For two roots $I,I'\in\cT$, the tree isomorphism
$T_{I,I'}:\cT(I)\to\cT(I'),\,I_\omega\mapsto I'_\omega$ sends every descendant of \(I\) to the corresponding descendant of \(I'\) at the same relative depth (distance to the root) and
preserves relative (to the root) measure.

In a general measure space $T_{I,I'}$ need not arise from
a pointwise bijection. However, for ultrametric local fields, there is a concrete pointwise realization. If
$B=b+t^k\mathcal O_K$ and $B'=b'+t^{k'}\mathcal O_K$, then
\[
 x\longmapsto b'+t^{k'-k}(x-b)
\]
is an affine bijection from $B$ to $B'$ preserving relative measures and
relative descendant depths. Labels can be chosen compatibly with this map.

\begin{defn}[Depth-$m$ functions]
A \emph{depth-$m$ function} on $I$ is constant on every member of $\cT_m(I)$;
the space of these restrictions is denoted by $\cA_m(I)$. Extensions outside
$I$ will always be specified separately, when necessary.
\end{defn}

\begin{defn}[Compatible local functionals]\label{def:scN}
For the remainder of this section, let 
$\scN_I:\BMO_\cT(I)\to[0,\infty]$ for $I
\in\cT$, and  define
\[
 \scN(f):=\sup_{I\in\cT}\scN_I(f|_I),\quad f\in\BMO_\cT(X).
\]
We assume the following properties:
\begin{enumerate}
\item \textbf{Positive Homogeneity:} $\scN_I(\lambda f)=\lambda\scN_I(f)$ for every $\lambda>0$ and $I\in\cT$, with $\scN_I(f)=0$, whenever $f$ is constant almost everywhere on $I$.

\item \textbf{Invariance under rooted tree isomorphisms:} Let $I,I'\in\cT$, and let $T_{I,I'}$ be a rooted tree isomorphism preserving relative depths. If $f\in\cA_m(I)$, write $T_{I,I'}f$ for the depth-\(m\) function having the same values as \(f\) on the
corresponding descendants. Then $\scN_{I'}(T_{I,I'}f)=\scN_I(f)$.
\end{enumerate}
\end{defn}
We refer to $\scN_I$ as the local functional and $\scN$ as its global envelope. This distinction matters even when the input is supported in a single node.

Define, with values allowed in $[0,\infty]$,
\[
 c_{\scN}:=\sup_{\|f\|_\ast\le1}\scN(f),\qquad
 c_{\scN}(I):=\sup_{\|f\|_{\ast,I}\le1}\scN_I(f).
\]
These constants are the norms of the corresponding positively
homogeneous functionals relative to the $\operatorname{L}^1$-BMO seminorm.
 Set
\[
 c_{\scN,m}(I):=\sup\{\scN_I(f):f\in\cA_m(I),\ \|f\|_{\ast,I}\le1\}.
\]
Equal relative leaf measures and invariance under rooted tree isomorphisms show that this
quantity is independent of $I$; write it as $c_{\scN,m}$. Every depth-$m$ function is also a depth-$m+1$ function, so
$c_{\scN,m}\le c_{\scN,m+1}$. We call $c_\scN$ the global sharp constant and $c_{\scN,m}$ the sharp constant at depth-$m$.

\begin{lemma}[Norm-preserving extension]\label{lem:extension}
If $f\in\BMO_\cT(I)$, define $\widehat f=f$ on $I$ and $\widehat f=f_I$
on $X\setminus I$. Then
\[
 \|\widehat f\|_\ast=\|f\|_{\ast,I},\quad \scN(\widehat f)\ge\scN_I(f).
\]
In particular, a centered local function has a norm-preserving zero extension.
\end{lemma}
\begin{proof}
Descendant oscillations are unchanged. A node disjoint from $I$ has zero
oscillation. If $P\supset I$, then $\widehat f_P=f_I$ and
\[
 \osc_P(\widehat f)=\frac{\mu(I)}{\mu(P)}\osc_I(f)
                   \le\|f\|_{\ast,I}.
\]
These are the only possible positions of a tree node relative to \(I\). Taking the supremum over
all tree nodes therefore gives equality.
\end{proof}
Related questions about transferring sharp BMO constants
between domains are studied by Stolyarov and Zatitskiy
\cite{StolyarovZatitskiy2021Transference}.
For the present tree setting, \cref{lem:extension} obtains
a norm-preserving extension directly from the nesting of
the averaging sets.

For $f\in \operatorname{L^1}(I)$ define finite-depth averages
\[
 E_m^I f:=\sum_{J\in\cT_m(I)} f_J\chi_J\in\cA_m(I).
\]
These can be seen as conditional expectations on the probability space
$(I,\mu/\mu(I))$. Since $\cT$ differentiates $\operatorname{L^1}(X,\cT)$, we have
\begin{equation}\label{Emf converges to f}
E_m^I f\longrightarrow f
\qquad
\text{a.e. on }I
\quad\text{and in }\operatorname{L^1}(I).
\end{equation}
Indeed, the local differentiation hypothesis follows from the global one by extending
an $\operatorname{L^1}(I)$ function by zero, which is tree-integrable. To obtain convergence in \(\operatorname{L}^1(I)\), put
\(f^{(M)}=\max\{-M,\min\{f,M\}\}\). The averaging operators
are contractions on \(\operatorname{L}^1(I)\), and
\(E_m^I f^{(M)}\to f^{(M)}\) in \(\operatorname{L}^1(I)\). Hence
\[
 \limsup_{m\to\infty}\|E_m^I f-f\|_{\operatorname{L}^1(I)}
 \le 2\|f-f^{(M)}\|_{\operatorname{L}^1(I)}.
\]
Letting \(M\to\infty\) proves \eqref{Emf converges to f}
\begin{lemma}[Contractive approximation]\label{lem:expectation}
Let $I\in\cT_k$. For every $f\in\BMO_\cT(I)$,
\[
 \|E_m^I f\|_{\ast,I}\le\|f\|_{\ast,I},\qquad
 \|E_m^I f\|_{\ast,I}\longrightarrow\|f\|_{\ast,I}.
\]
\end{lemma}
\begin{proof}
We can assume that $k=0$. On a descendant $P$ above depth $m$, the average is unchanged and Jensen's
inequality gives $\osc_P(E_m^I f)\le\osc_P(f)$. Indeed, if \(P\) has depth \(r\le m\), its depth-\(m\) descendants partition \(P\), and
\(E_m^I f=f_J\) on each such descendant \(J\). Since their weighted average
is \(f_P\), Jensen's inequality gives
\[
\osc_P(E_m^I f)
=\frac1{\mu(P)}\sum_{J\in\mathcal T_{m-r}(P)}
  \mu(J)|f_J-f_P|
\le \frac1{\mu(P)}\int_P|f-f_P|\,d\mu
=\osc_P(f).
\]
On a deeper descendant the
step function is constant. This proves contraction. For a fixed \(P\in\mathcal T(I)\), the \(\operatorname{L^1}(P)\) convergence
\(E_m^I f\to f\) and the estimate
\begin{equation}\label{eq:osc-continuity}
\left|\osc_P(g)-\osc_P(h)\right|
\le \frac{2}{\mu(P)}\int_P|g-h|\,d\mu
\end{equation}
imply that \(\osc_P(E_m^I f)\to\osc_P(f)\). Given \(\varepsilon>0\), choose
\(P\) such that \(\osc_P(f)>\|f\|_{*,I}-\varepsilon\). Since
\(\|E_m^I f\|_{*,I}\ge\osc_P(E_m^I f)\), we obtain
\[
\liminf_{m\to\infty}\|E_m^I f\|_{*,I}
\ge \osc_P(f)>\|f\|_{*,I}-\varepsilon.
\]
Combining this with the contraction
\(\|E_m^I f\|_{*,I}\le\|f\|_{*,I}\), proves
$\|E_m^I f\|_{*,I}\longrightarrow\|f\|_{*,I}$.
\end{proof}

\begin{thm}[Local, global, and finite-depth constants]\label{thm:reduction}
Suppose the local functionals satisfy \cref{def:scN} and
\begin{equation*}\tag{$\star$}\label{eq:LSC}
 \scN_I(f)\le\liminf_{m\to\infty}\scN_I(E_m^I f),
 \quad \text{ for every } f\in\BMO_\cT(I).
\end{equation*}
Then for every $I$,
\begin{equation}\label{eq:all-constants}
 c_{\scN}=c_{\scN}(I)=\sup_{m\ge1}c_{\scN,m}.
\end{equation}
\end{thm}
\begin{proof}
Put $C=\sup_m c_{\scN,m}$. The class defining $c_\scN(I)$ contains the class defining $c_{\scN,m}$ for every $m$,
so $C\le c_{\scN}(I)$. If \(C=\infty\), the asserted equalities follow immediately.
Assume that \(C<\infty\). Let $f\in\BMO_\cT(I)$. For each \(m\), \(E_m^I f\) has depth \(m\) and so, by homogeneity and the
definition of \(c_{\scN,m}\),
\[
\scN_I(E_m^I f)
\le c_{\scN,m}\|E_m^I f\|_{*,I}
\le C\|f\|_{*,I},
\]
where the last inequality uses \cref{lem:expectation}, so long as $\|E_m^I f\|_{*,I}\neq 0$. This inequality remains true when
\(\|E_m^I f\|_{*,I}=0\), by the definition of $\scN_I$ since \(E_m^I f\) is constant almost everywhere. Assumption
\eqref{eq:LSC} now gives
\[
\scN_I(f)
\le\liminf_{m\to\infty}\scN_I(E_m^I f)
\le C\|f\|_{*,I}.
\]
Thus $c_{\scN}(I)=C$.

For $f\in\BMO_\cT(X)$, $\scN(f)\le C\|f\|_\ast$ by taking the supremum of the
local inequalities. Conversely, extend any locally admissible $f$ by
\cref{lem:extension}. The extension has the same input seminorm, while
$\scN(\widehat f)\ge\scN_I(f)$. Hence $c_{\scN}\ge c_{\scN}(I)$.
\end{proof}
\begin{remark}\label{rem:lsc-only}
For nonlinear composition, $|E_m f|$ and $E_m|f|$ are different functions.
The contraction of conditional expectation cannot be applied to identify
their BMO seminorms. Condition \eqref{eq:LSC} avoids this issue. No assertion
of BMO convergence of $E_m f$ to $f$ is needed.
\end{remark}

\subsection{Leaf-value formulas for depth-\texorpdfstring{$m$}{m} functions}
\label{subsec:leaves}
Fix $I\in\cT$. Let $\Omega_{m,q}:=\{1,\ldots,q\}^m$, $V_{m,q}:=\R^{\Omega_{m,q}}$, and
$\one=(1,\ldots,1)$. For a prefix $\sigma$ let $I_\sigma=\{\omega:\omega\mid_{|\sigma|}=\sigma\}$ be the set of leaves
below it and $n_\sigma=|I_\sigma|=q^{m-|\sigma|}$. Here, we also use \(I_\sigma\) to denote the node containing depth-$m$ descendants indexed by those leaves. To avoid cumbersome notation, we adopt this convention and specify the intended meaning whenever necessary. Note that each node \(I_\sigma\) is uniquely determined by the set of leaves below it.

Empty prefix $\emptyset$ corresponds to the root $I$. For $a\in V_{m,q}$ set
\[
 A_\sigma(a):=\frac1{n_\sigma}\sum_{\omega\in I_\sigma}a_\omega,
 \qquad D_\sigma(a):=\frac1{n_\sigma}
          \sum_{\omega\in I_\sigma}|a_\omega-A_\sigma(a)|,
\]
\[
 F_{m,q}(a):=\max_{|\sigma|\le m}D_\sigma(a).
\]
For the depth-$m$ function $f_a$ with leaf values $a$, we have $(f_a)_{I_\sigma}=A_\sigma(a)$ and $\|f_a\|_{\ast,I}=F_{m,q}(a)$.
The seminorm $F_{m,q}$ has kernel $\R\one$. Define
\[
 \Sigma_{m,q}:=\{F_{m,q}\le1\},\qquad
 \overline\Sigma_{m,q}:=\{a\in\Sigma_{m,q}:A_\emptyset(a)=0\}, \qquad \scN_m:V_{m,q}\to [0,\infty],\, a\mapsto\scN_I(f_a).
\]
Then
\begin{equation}\label{eq:finite-optimization}
 c_{\scN,m}=\sup_{a\in\Sigma_{m,q}}\scN_m(a).
\end{equation}
This constant is independent of the choice of $I$.
A supremum is used here; its attainment requires an additional hypothesis, which we discuss in \Cref{subsec:anchors}.
If $\scN$ is translation invariant, the supremum can be restricted to
$\overline\Sigma_{m,q}$.

\begin{lemma}[Edge and amplitude estimates]\label{lem:edge}
For a child $J$ of $I$,
\begin{equation}\label{eq:edge-average-bound}
 |f_J-f_I|\le\frac q2\osc_I(f).
\end{equation}
Consequently, if $F_{m,q}(a)\le1$, then
\[
 |a_\omega-A_\emptyset(a)|\le\frac{qm}{2}.
\]
If the value of one leaf is zero, $|A_\emptyset(a)|\le qm/2$ and $|a_\omega|\le qm$.
\end{lemma}
\begin{proof}
Write $d_j=f_{I_j}-f_I$. Their sum is zero and
$q^{-1}\sum_j|d_j|\le\osc_I(f)$ by Jensen's inequality.
A zero-sum vector satisfies $|d_j|\le\frac12\sum_i|d_i|$.
This proves \eqref{eq:edge-average-bound}. Summation along a root-to-leaf
chain gives the remaining statements.
\end{proof}

\subsection{Finite anchoring and rational polyhedral extremizers}
\label{subsec:anchors}
Since \(F_{m,q}(a+c\mathbf 1)=F_{m,q}(a)\), the admissible set is
unbounded in the constant direction. To obtain compact optimization
problems, we seek a representative from each translation class without
decreasing the functional being maximized. The next definition formalizes
this through finitely many anchors: each anchor selects one point on a
translation line, and the functional's supremum over that line is captured
by these finitely many choices. The condition
\(\ell_j(\mathbf 1)\ne0\) ensures that each anchor hyperplane
\(\ell_j(a)=0\) meets every translation line exactly once.

A closed rational polyhedral cone is a region in $\R^N$ defined by an intersection of finitely many homogeneous linear inequalities with rational coefficient. A ploytope is a bounded polyhedron.
\begin{defn}\label{def:anchors}
A continuous functional $G:V_{m,q}\to[0,\infty)$ has \emph{finite anchoring} if there are linear forms $\ell_1,\ldots,\ell_N$ with
$\ell_j(\one)\ne0$ such that
\begin{equation}\label{eq:anchors}
 \sup_{c\in\R}G(a+c\one)=
 \max_{1\le j\le N}G\left(a-\frac{\ell_j(a)}{\ell_j(\one)}\one\right).
\end{equation}
It is \emph{rational piecewise linear} if it is linear with rational
coefficients on each member of a finite covering by closed rational
polyhedral cones. Anchors are rational when their coefficients are rational.
\end{defn}
``Cone'' means that if $x$ is in the region, then so is $tx$ for every $t\geq 0$.
\begin{lemma}[Compact transverse sections]\label{lem:compact}
If $\ell(\one)\ne0$, then
$P_\ell:=\Sigma_{m,q}\cap\ker\ell$ is compact. If $\ell$ is rational,
$P_\ell$ is a rational polytope.
\end{lemma}
\begin{proof}
The restriction of $F_{m,q}$ to $\ker\ell$ is a norm: its kernel is the
intersection of $\ker\ell$ with $\R\one$, which is zero.
Its unit ball is compact by finite dimensionality. More explicitly, the
minimum of $F_{m,q}$ on the Euclidean unit sphere of $\ker\ell$ is positive,
so $P_\ell$ is bounded and closed. Moreover,
\[
 D_\sigma(a)=\max_{\varepsilon\in\{-1,1\}^{I_\sigma}}
 \frac1{n_\sigma}\sum_{\omega\in I_\sigma}
       \varepsilon_\omega(a_\omega-A_\sigma(a)).
\]
Thus $F_{m,q}$ is a maximum of finitely many rational linear forms.
Intersecting its unit ball with a rational hyperplane gives the asserted
rational polytope.
\end{proof}

\begin{thm}[Finite-anchor extremizers]\label{thm:finite-anchor}
If $G$ is continuous and satisfies \eqref{eq:anchors}, then
\[
 \sup_{\Sigma_{m,q}}G=\max_{1\le j\le N}\max_{P_{\ell_j}}G.
\]
In particular, the supremum is attained. If $G$ is rational piecewise linear
and the anchors are rational, there is a rational extremizer and the maximum
is rational.
\end{thm}
\begin{proof}
Translations preserve $F_{m,q}$. Equation \eqref{eq:anchors} therefore moves
every admissible vector to one of the compact sections $P_{\ell_j}$ without
decreasing $G$. Compactness and continuity give attainment. In the rational
case, intersect each section with each linearity cone of $G$.
Every nonempty intersection is a compact rational polytope. A linear
functional attains its maximum at a vertex, every such vertex is rational,
and its value under a rational linear form is rational. There are finitely
many intersections.
\end{proof}

\begin{remark}[Exact certificates]\label{rem:certificates}
The proof gives a finite collection of rational linear programs. An upper
bound can be certified by feasible dual solutions on all cells, and a
matching rational primal vector certifies equality. Numerical optimization
can suggest a cell or an extremizer. The distinction between fixed depth and all depths
is equally important: a supremum of rational finite-depth constants need not
be rational or attained at a finite depth.
\end{remark}

\subsection{Value spectra}\label{subsec:spectra}
A leaf \(\omega\in\Omega_{m,q}\) is a \emph{zero leaf} of \(a\)
if its value is zero, that is, if \(a_\omega=0\).
\begin{defn}[Zero-leaf anchoring]\label{def:zero-leaf}
A family $\{\scN_m\}$ has \emph{zero-leaf anchoring} if
\[
 \sup_{c\in\R}\scN_m(a+c\one)
 =\max_{\omega\in\Omega_{m,q}}\scN_m(a-a_\omega\one)
\]
for every $m$ and $a$. In other words, for every \(m\ge1\), $\scN_m$ has the finite anchoring with anchors $\ell_\omega(a)=a_\omega$.
\end{defn}
Thus an optimizing translation can be chosen by subtracting one of
the leaf values: for some leaf \(\omega\), the translated vector
\(a-a_\omega\one\) has a zero entry and attains the supremum over
all translations.
\begin{cor}[Zero-leaf extremizer]\label{cor:zero-leaf-extremizer}
Assume that \(\{\scN_m\}\) is continuous and has the zero-leaf anchoring. Then, for every \(m\ge1\) and every fixed leaf
\(\omega_0\in\Omega_{m,q}\),
\[
c_{\scN,m}
=
\max\left\{
\scN_{m}(a):
a\in\Sigma_{m,q},\ a_{\omega_0}=0
\right\}.
\]
\end{cor}
\begin{proof}
For continuous functionals, \cref{thm:finite-anchor} gives attainment.
Invariance under rooted tree automorphisms allows any prescribed leaf to
serve as the zero leaf, because these automorphisms act transitively on the
leaves. No pointwise affine self-map of the original measure space is needed.
\end{proof}
\begin{prop}[The value spectrum]\label{prop:spectrum}
Suppose $\scN_m$ is continuous, positively homogeneous, and finitely anchored.
Then
\[
 \{\scN_m(a):a\in\Sigma_{m,q}\}=[0,c_{\scN,m}].
\]
Consequently, consecutive finite-depth constants coincide exactly when their
value spectra coincide.
\end{prop}
\begin{proof}
Let $a^\ast$ attain $c_{\scN,m}$. If this constant is positive, the vectors
$ta^\ast$, $0\le t\le1$, are admissible and have values $tc_{\scN,m}$.
If it is zero, the assertion is immediate. No value can exceed the maximum.
\end{proof}

\subsection{Stable centering and polarization}\label{subsec:centering}
\begin{defn}[Heredity]\label{def:hereditary}
The family $\{\scN_m\}$ is hereditary under descendant embeddings if
$\scN_M(\widetilde a)\ge\scN_m(a)$ whenever $a$ occurs unchanged as the leaf
configuration of a depth-$m$ descendant subtree of $\widetilde a\in V_{M,q}$.
\end{defn}
Define
\[
\bar c_{\scN,m}:=\sup_{a\in\overline\Sigma_{m,q}}\scN_m(a),\quad \bar c_{\scN}:=\sup_m \bar c_{\scN,m}
\]
Assume the functionals $\scN_m$ are continuous and positively homogeneous.
By \cref{lem:compact}, the centered
constant
\[
 \bar c_{\scN,m}
=\max_{a\in\overline\Sigma_{m,q}}\scN_m(a)
\]
is attained. These constants are nondecreasing.
\begin{thm}[Stable centering]\label{thm:stable-centering}
Suppose that the family \(\{\scN_m\}\) satisfies zero-leaf
anchoring and heredity. For every \(m\ge1\) and every
\(a\in\Sigma_{m,q}\), there exist an integer \(L\) and a
centered admissible vector \(\widetilde a\in\) such that
\[
 0\le L\le\left\lceil\frac{qm}{q-1}\right\rceil,
 \qquad
 \widetilde a\in\overline\Sigma_{m+L,q},
 \qquad
 \scN_{m+L}(\widetilde a)\ge\scN_m(a).
\]
Consequently,
\begin{align}
 c_{\scN,m}
 &\le
 \bar c_{\scN,m+\lceil qm/(q-1)\rceil},
 \label{eq:center-depth}\\
 \sup_{m\ge1}c_{\scN,m}
 &=\sup_{m\ge1}\bar c_{\scN,m}.
 \label{eq:center-sup}
\end{align}
If \(L\ge1\), the vector \(\widetilde a\) can be chosen
so that one child of its root has
mean \(\mu\), while each of the other \(q-1\) children
is constant with value \(-\mu/(q-1)\).
Then child averages are
\begin{equation}\label{eq:polarization}
 \mu,-\frac{\mu}{q-1},\ldots,-\frac{\mu}{q-1}.
\end{equation}
\end{thm}
\begin{proof}
By zero-leaf anchoring, translate \(a\) so that the translated vector has a zero leaf and its
functional value does not decrease. Denote the resulting vector by \(b^{(0)}=a-a_{\omega_0}\one\).
It remains admissible, and its mean \(\mu_0\) satisfies
$|\mu_0|=|A_\emptyset(b^{(0)})|\le \frac{qm}{2}$
by \cref{lem:edge}. Given $b^{(r)}$ with mean
$\mu_r\ne0$, put
\[
 \delta_r=\min\{|\mu_r|,(q-1)/2\},\quad
 \mu_{r+1}=\mu_r-\sgn(\mu_r)\delta_r,\quad
 \kappa_r=\frac{q\mu_{r+1}-\mu_r}{q-1}.
\]
Form a new root whose first child carries \(b^{(r)}\) and whose other
\(q-1\) children are constant, each with value \(\kappa_r\). Call the
resulting vector \(b^{(r+1)}\). Its root mean is
$\frac{\mu_r+(q-1)\kappa_r}{q}=\mu_{r+1}$. Old descendant oscillations
remain unchanged, and constant siblings have zero oscillation. At the new
root, using normalized averaging on the distinguished child,
\begin{align*}
 D_{\emptyset}(b^{(r+1)})
 =\frac1q\E|b^{(r)}-\mu_{r+1}|
        +\frac{q-1}{q}|\kappa_r-\mu_{r+1}|
 \le \frac{D_{\emptyset}(b^{(r)})+\delta_r}{q}
  +\frac{\delta_r}{q}
\le \frac{1+2\delta_r}{q}
\le 1.
\end{align*}
Here the expectation is with respect to normalized measure on the retained
child. Thus \(b^{(r+1)}\) is admissible whenever \(b^{(r)}\) is admissible. At each step, the magnitude of the mean decreases by
\(\min\{|\mu_r|,(q-1)/2\}\). Hence the construction reaches mean zero after
$L=\left\lceil\frac{2|\mu_0|}{q-1}\right\rceil$
steps, with \(L=0\) if \(\mu_0=0\).
Heredity ensures that the functional
value does not decrease during these extensions. Since
$L\le \left\lceil\frac{qm}{q-1}\right\rceil$,
constant refinement, if needed, gives a centered vector at depth
\(m+\lceil qm/(q-1)\rceil\). This proves \eqref{eq:center-depth}.
The centered vectors are a subclass of all admissible vectors, so
\eqref{eq:center-sup} follows. If \(L\ge1\), then at the last extension \(\mu_{r+1}=0\). The retained
child has mean \(\mu_r\), and the other children have mean
\(\kappa_r=-\mu_r/(q-1)\). This proves \eqref{eq:polarization}.
\end{proof}
\begin{cor}[General centering with finite anchors]\label{cor:general-centering}
Suppose the family $\{\scN_m\}$ is hereditary and satisfies finite anchoring at
depth \(m\), \textit{i.e.}, $\scN_m$ has finite anchoring with linear forms \(\ell_1,\ldots,\ell_N\). Set
\[
 R_j:=\max\bigl\{|A_\emptyset(b)|:
              b\in\Sigma_{m,q},\ \ell_j(b)=0\bigr\},
 \qquad R(m):=\max_{1\le j\le N}R_j.
\]
Then every value attained by an admissible depth-\(m\) vector is bounded
above by a value attained by a centered admissible vector at some
depth at most \(m+\lceil 2R(m)/(q-1)\rceil\). If finite anchoring
holds at every depth, then
\[
 \sup_m c_{\scN,m}=\sup_m\bar c_{\scN,m}.
\]

\end{cor}

\begin{proof}
Each \(R_j\) is finite by \cref{lem:compact}. Given
\(a\in\Sigma_{m,q}\), finite anchoring provides an index \(j\) for
which
\[
 b:=a-\frac{\ell_j(a)}{\ell_j(\one)}\one
 \quad\text{satisfies}\quad
 \ell_j(b)=0,\qquad \scN_m(b)\ge\scN_m(a).
\]
Translation leaves \(F_{m,q}\) unchanged, so \(b\) is admissible.
By the definition of \(R(m)\), its root mean has absolute value at
most \(R(m)\). Apply the extension construction in the proof of
\cref{thm:stable-centering} to \(b\). That construction reaches mean
zero after at most \(\lceil 2R(m)/(q-1)\rceil\) added generations,
while preserving admissibility and the lower bound on the
functional. If finite anchoring holds at every depth, take
supremum over \(m\) to obtain the stated equality.
\end{proof}
\begin{remark}[What centering preserves]\label{rem:centering-limits}
The old vector remains unchanged inside a descendant subtree. The new root
has mean zero, but the retained subtree need not be centered. In
\eqref{eq:polarization}, only the child \emph{averages} have opposite signs;
the distinguished child can contain leaves of both signs. Consequently, an
estimate just at the new root does not bound a hereditary maximum over all
descendants. An induction must control the class of subtrees that actually
occurs. The binary example in \cref{thm:binary-counterexample} makes this
distinction explicit.
\end{remark}

\begin{remark}[\(\operatorname{L}^p\)-BMO]
For \(1\le p<\infty\), one may replace \(\osc_I(f)\) throughout the
finite-tree setup by
\[
 \osc_{p,I}(f)
 :=\left(\avg{|f-f_I|^p}_I\right)^{1/p}.
\]
The extension, approximation, and finite-depth reduction arguments
apply with this seminorm. Finite anchoring still removes the
constant direction, and the centering construction works with a
step size depending on \(p\) and \(q\). Indeed, suppose a retained child carries a function with
mean \(\mu\) and \(\operatorname{L}^p\)-BMO seminorm at most
\(1\). To give its new parent \(P\) mean \(\mu'\), fill the
other children constantly with
 $\kappa=\frac{q\mu'-\mu}{q-1}$.
Writing \(\delta=|\mu-\mu'|\), Minkowski's inequality gives
\[
 \osc_{p,P}(\widetilde f)
 \le q^{-1/p}
 \left(1+\delta\bigl[1+(q-1)^{1-p}\bigr]^{1/p}\right).
\]
Thus the new parent remains admissible whenever
\[
 \delta\le
 \frac{q^{1/p}-1}{\bigl[1+(q-1)^{1-p}\bigr]^{1/p}}.
\]
Moving the mean toward zero by steps of at most this size
proves the centering assertion. Thus the structural parts
of the formalism have \(\operatorname{L}^p\) analogues.

The rational-polytope conclusion requires separate attention.
For \(p>1\), the constraint \(F_{m,q}^{(p)}(a)\le1\) is generally
not polyhedral, so rational piecewise-linearity of an output
functional does not by itself imply a rational extremal value.
The applications below use \(p=1\), where the sharp constants
and the distinction between node and leaf anchoring are central.
\end{remark}

\section{Jensen-defect Carleson functionals}\label{sec:jensen}
For a convex function \(\Phi\), the Jensen gap measures the
difference between the average of \(\Phi(f)\) and its value
at the average of \(f\). On a tree, this gap decomposes into
nonnegative one-step defects associated with successive
generations. We show that their weighted descendant sums
recover the gap exactly. These sums have the form of a
Carleson packing condition, whose classical background is
\cite{Carleson1962}; the dyadic descendant-sum condition
appears explicitly in
\cite[Theorem~4.5]{HytonenPerez2013Sharp}.
The resulting functional illustrates node anchoring.
For \(\Phi(t)=|t|\), its defects measure cancellation between
child averages and connect this example to the
absolute-value map studied in the next section.

Fix $B\in\cT$ and a convex Lipschitz function
$\Phi:\R\to\R$. For $f\in \operatorname{L^1}(B)$ and $P\in\cT(B)$, define the
\emph{Jensen gap} on $P$ by
\[
J_P^\Phi(f)
:=
\frac1{\mu(P)}\int_P\Phi(f)\,d\mu-\Phi(f_P)\ge0.
\]
Set
\[
\scC_{\Phi,B}(f):=\sup_{P\in\cT(B)}J_P^\Phi(f),
\qquad
\scC_\Phi(f):=\sup_{B\in\cT}\scC_{\Phi,B}(f).
\]

For a node $I$ with children $I_1,\ldots,I_q$, define its
\emph{one-step Jensen defect} by
\begin{equation}\label{eq:lambdaI}
\lambda_I^\Phi(f)
:=
\frac1q\sum_{j=1}^q\Phi(f_{I_j})-\Phi(f_I)\ge0.
\end{equation}
The inequality follows from
$f_I=q^{-1}\sum_{j=1}^q f_{I_j}$ and the convexity of $\Phi$.


For a nonnegative sequence $\lambda=\{\lambda_I\}$ on the
descendants of $B$, define its local Carleson norm by
\[
\|\lambda\|_{\Carl(B)}
:=
\sup_{P\in\cT(B)}
\frac1{\mu(P)}
\sum_{I\in\cT(P)}\lambda_I\mu(I).
\]
The sum includes $P$ and is the increasing limit of its finite
generation sums. We say that $\lambda$ is a Carleson sequence if its local Carleson norm is finite. 

\subsection{An exact packing identity}\label{subsec:telescope}

\begin{prop}[Exact Jensen-defect packing]\label{thm:packing}
For every $f\in \operatorname{L^1}(B)$ and $P\in\cT(B)$,
\begin{equation}\label{eq:packing}
\sum_{I\in\cT(P)}\lambda_I^\Phi(f)\mu(I)
=
\int_P\Phi(f)\,d\mu-\mu(P)\Phi(f_P)
=
\mu(P)J_P^\Phi(f).
\end{equation}
Consequently,
\begin{equation}\label{eq:C-gap}
\scC_{\Phi,B}(f)
=
\|\lambda^\Phi(f)\|_{\Carl(B)},
\end{equation}
and if $\Phi$ is $L$-Lipschitz, then
\[
\scC_{\Phi,B}(f)\le L\|f\|_{*,B}.
\]
\end{prop}

\begin{proof}
Since $\mu(I_j)=\mu(I)/q$, multiplication of the one-step defect
by $\mu(I)$ gives
\[
\mu(I)\lambda_I^\Phi(f)
=
\sum_{j=1}^q\mu(I_j)\Phi(f_{I_j})
-\mu(I)\Phi(f_I).
\]
On summing over the first $N$ generations below $P$, every term
at an intermediate generation cancels. Hence, writing $\operatorname{depth}_P(I)$ for the depth of $I$ in $\cT(P)$,  
\begin{equation}\label{eq:finite-sum-in-Carleson-id}
\sum_{\substack{I\in\cT(P)\\
                \operatorname{depth}_P(I)<N}}
\lambda_I^\Phi(f)\mu(I)
=
\sum_{J\in\cT_N(P)}\mu(J)\Phi(f_J)
-\mu(P)\Phi(f_P)
=
\int_P\Phi(E_N^P f)\,d\mu-\mu(P)\Phi(f_P).
\end{equation}
The finite-depth averages $E_N^P f$ converge to $f$ in $\operatorname{L^1}(P)$, since $\cT(P)$ differentiates $\operatorname{L^1}(P)$.
Since $\Phi$ is $L$-Lipschitz,
\[
\int_P|\Phi(E_N^P f)-\Phi(f)|\,d\mu
\le L\int_P|E_N^P f-f|\,d\mu\longrightarrow0.
\]
All one-step defects are nonnegative, so the finite sums \eqref{eq:finite-sum-in-Carleson-id} increase
to the sum in \eqref{eq:packing}. Dividing that identity by
$\mu(P)$ and taking the supremum over $P$ proves
\eqref{eq:C-gap}. Finally,
\[
J_P^\Phi(f)
=
\frac1{\mu(P)}
\int_P\bigl(\Phi(f)-\Phi(f_P)\bigr)\,d\mu
\le L\frac1{\mu(P)}\int_P|f-f_P|\,d\mu.
\]
Taking the supremum over $P$ proves the final inequality.
\end{proof}

Thus the Carleson estimate follows from an identity: the total
one-step defect below $P$ equals the Jensen gap on $P$. At finite
depth, the functional takes the form
\begin{equation}\label{eq:C-finite}
\scC_{\Phi,m}(a)
=
\max_\sigma
\left(
\frac1{n_\sigma}\sum_{\omega\in I_\sigma}\Phi(a_\omega)
-\Phi(A_\sigma(a))
\right).
\end{equation}
This formula shows that finite-depth tree transport preserves the
functional and that embedding a vector as a descendant subtree
cannot decrease its value.

The finite-depth approximations satisfy
\[
\scC_{\Phi,B}(E_m^B f)\longrightarrow\scC_{\Phi,B}(f).
\]
Indeed, if $P$ lies above depth $m$, then
$(E_m^Bf)_P=f_P$, and Jensen's inequality on the depth-$m$
partition of $P$ gives
$J_P^\Phi(E_m^Bf)\le J_P^\Phi(f)$. If $P$ lies at or below
depth $m$, then $E_m^Bf$ is constant on $P$ and its gap there is
zero. This proves
$\scC_{\Phi,B}(E_m^Bf)\le\scC_{\Phi,B}(f)$.
For each fixed $P$, the $\operatorname{L^1}(P)$ convergence and the Lipschitz
property of $\Phi$ give
$J_P^\Phi(E_m^Bf)\to J_P^\Phi(f)$. Choosing a fixed $P$
arbitrarily close to the supremum proves the reverse limit
inequality.

These statements hold for every convex Lipschitz $\Phi$.
To apply the homogeneous reduction in \cref{thm:reduction}, the
functional must in addition be positively homogeneous. This holds,
for example, when $\Phi$ is positively homogeneous, in which case
$\Phi(t)=a t+\kappa|t|$ for some $a\in\R$ and $\kappa\ge0$ as we will see below. The affine terms
cancel from every Jensen gap, leaving
$\scC_{\Phi,B}=\kappa\scC_{|\cdot|,B}$.

\subsection{Homogeneous defects and node anchoring}
For a measurable set $P$ with $0<\mu(P)<\infty$, we may write
\[
\avg{f}_P:=f_P=\frac{1}{\mu(P)}\int_P f\,d\mu
\]
for the average of $f$ over $P$.

A convex positively homogeneous function on $\R$ has the form
\begin{equation}\label{eq:Phi-classification}
 \Phi(t)=a_\Phi t+\kappa_\Phi|t|,
 \quad a_\Phi=\frac{\Phi(1)-\Phi(-1)}2,
 \quad \kappa_\Phi=\frac{\Phi(1)+\Phi(-1)}2\ge0.
\end{equation}
This follows by homogeneity on the positive and negative half-lines, and convexity at zero implies that $0=\Phi(0)\le\kappa_\Phi$.
Since affine terms cancel, $\scC_\Phi=\kappa_\Phi\scC_{|\cdot|}$.
Write $\scC_B=\scC_{|\cdot|,B}$ and $\scC_m=\scC_{|\cdot|,m}$.

\begin{thm}[Optimized packing and exact node anchoring]\label{thm:node-anchor}
For every $f\in\BMO_\cT(B)$,
\begin{equation}\label{eq:optimized-C}
 \sup_{c\in\R}\scC_B(f+c)=\|f\|_{\ast,B}.
\end{equation}
At finite depth,
\begin{equation}\label{eq:node-anchor}
 \sup_{c\in\R}\scC_m(a+c\one)
 =\max_\sigma\scC_m(a-A_\sigma(a)\one)=F_{m,q}(a).
\end{equation}
For a homogeneous $\Phi$ as in \eqref{eq:Phi-classification},
\[
 c_{\scC_\Phi,m}=\bar c_{\scC_\Phi,m}=c_{\scC_\Phi}=\kappa_\Phi
 \quad(m\ge1),
\]
and its finite-depth value spectrum is $[0,\kappa_\Phi]$.
\end{thm}
\begin{proof}
For every descendant $P$ of $B$ and $c\in\R$,
\[
 \avg{|f+c|}_P-|f_P+c|\le\avg{|f-f_P|}_P\le\|f\|_{\ast,B}.
\]
Choosing $c=-f_P$ gives equality on the first inequality. Taking supremum over $c$ and
$P$ proves \eqref{eq:optimized-C}. At finite depth, write
\[
 J_\sigma(a;c)
 :=\frac1{n_\sigma}\sum_{\omega\in I_\sigma}|a_\omega+c|
       -|A_\sigma(a)+c|.
\]
The preceding inequality, applied at \(I_\sigma\), says that
\(J_\sigma(a;c)\le D_\sigma(a)\) for every \(c\). Choosing
\(c=-A_\sigma(a)\) gives \(J_\sigma(a;c)=D_\sigma(a)\).
Since there are finitely many nodes,
\[
 \sup_{c\in\R}\scC_m(a+c\one)
 =\sup_{c\in\R}\max_\sigma J_\sigma(a;c)
 =\max_\sigma D_\sigma(a)
 =F_{m,q}(a).
\]
For every node \(\sigma\),
\(\scC_m(a-A_\sigma(a)\one)\le F_{m,q}(a)\).
If \(\sigma\) maximizes \(D_\sigma(a)\), its contribution to
\(\scC_m(a-A_\sigma(a)\one)\) equals \(F_{m,q}(a)\). Hence
 $\max_\sigma\scC_m(a-A_\sigma(a)\one)=F_{m,q}(a)$,
showing \eqref{eq:node-anchor}. The finitely many linear forms
\(\ell_\sigma(a):=A_\sigma(a)\) satisfy
\(\ell_\sigma(\one)=1\), so \eqref{eq:node-anchor} gives finite
anchoring in the sense of \cref{def:anchors}.

For sharpness, put $(q-1,-1,\ldots,-1)$ on the children of a root. Its mean
is zero, and both its BMO seminorm and its absolute-value Jensen gap $\scC_B$ are
$2(q-1)/q$. We normalize and refine constantly to obtain a centered extremizer
at every depth. More precisely, for any \(m\ge1\), multiply the \(q\) values by
\(\frac{q}{2(q-1)}\), and assign each resulting value to every
depth-\(m\) leaf below its corresponding root child. The resulting
vector has root average zero and satisfies
\(F_{m,q}(a)=\scC_m(a)=1\). Since
\(\scC_\Phi=\kappa_\Phi\scC_{|\cdot|}\), the finite-depth sharp
constant $c_{\scC_\Phi,m}$ is \(\kappa_\Phi\). The global equality follows from
\cref{thm:reduction}, and positive homogeneity gives the stated
value spectrum by \cref{prop:spectrum}.
\end{proof}
We call \eqref{eq:node-anchor} \emph{node anchoring} because the
translation that maximizes \(\scC_m\) can be chosen to make the
average on some node \(\sigma\) equal to zero. The node need not be a leaf. For example, the depth-one
vector $a=(q-1,-1,\ldots,-1)$
has root average zero and Jensen gap \(2(q-1)/q>0\).
Subtracting either of its two distinct leaf values makes
all entries nonnegative or all entries nonpositive, so
every Jensen gap becomes zero. Repeating each value on
its descendants gives the same example at every positive
depth. Thus \(\scC_m\) has node anchoring but does not have
zero-leaf anchoring. \cref{cor:general-centering} applies through its more
general finite-anchoring hypothesis.

For these node anchors, \cref{cor:general-centering} has an explicit
mean bound. If the selected node \(\sigma\) has depth \(r\), then
any admissible vector \(b\) with \(A_\sigma(b)=0\) satisfies
\[
 |A_\emptyset(b)|
 =|A_\emptyset(b)-A_\sigma(b)|
 \le \sum_{k=0}^{r-1}
       |A_{\sigma|_k}(b)-A_{\sigma|_{k+1}}(b)|
 \le \frac{qr}{2},
\]
where \(\sigma|_k\) is the ancestor of \(\sigma\) at depth \(k\).
The last inequality follows from \eqref{eq:edge-average-bound}.
Thus the constant \(R_j\) in \cref{cor:general-centering} is at most
\(qr/2\) when \(\ell_j=A_\sigma\).

\subsection{The sharp constant for general convex Lipschitz functions}
In this subsection, \(\Phi:\mathbb R\to\mathbb R\) is convex and
Lipschitz. No homogeneity assumption is imposed.
The Lipschitz constant in \cref{thm:packing} can be reduced by removing the
best affine slope. Let
\[
 s_-:=\lim_{t\to-\infty}\Phi'_+(t),\qquad
 s_+:=\lim_{t\to+\infty}\Phi'_+(t),\qquad
 \kappa(\Phi):=\frac{s_+-s_-}{2},
\]
where $\Phi'_+$ is the right derivative of $\Phi$. Convexity makes the one-sided derivative nondecreasing, and Lipschitz
continuity makes both limits finite.
\begin{thm}[Sharp slope-variation constant]\label{thm:slope-variation}
For every convex Lipschitz function \(\Phi\), every
\(B\in\cT\), and every \(f\in\BMO_\cT(B)\),
\begin{equation}\label{eq:sharp-general-packing}
 \scC_{\Phi,B}(f)\le\kappa(\Phi)\|f\|_{\ast,B}.
\end{equation}
The constant \(\kappa(\Phi)=(s_+-s_-)/2\) is optimal on every
homogeneous \(q\)-ary tree. Sharpness is obtained in the limit by
functions that are constant on the children of \(B\).
\end{thm}

\begin{proof}
Set
\[
\Psi(t)=\Phi(t)-\frac{s_++s_-}{2}t.
\]
The slopes of \(\Psi\) lie between \(-\kappa(\Phi)\) and
\(\kappa(\Phi)\), so \(\Psi\) is \(\kappa(\Phi)\)-Lipschitz.
Subtracting an affine function leaves every Jensen defect unchanged.
The bound in \cref{thm:packing}, applied to \(\Psi\), gives
\eqref{eq:sharp-general-packing}.

For sharpness, let \(f_R\) equal \(R(q-1)\) on one child of \(B\)
and \(-R\) on each of the other \(q-1\) children, where \(R>0\).
Then \((f_R)_B=0\), all oscillations below the first generation
vanish, and
$\|f_R\|_{\ast,B}=2R(q-1)/q$. Its Jensen gap at \(B\) is
\[
\avg{\Phi(f_R)}_B-\Phi(0)
=
\frac{\Phi(R(q-1))+(q-1)\Phi(-R)}q-\Phi(0).
\]
As \(R\to\infty\), the limiting slopes give
\[
\frac{\Phi(R(q-1))}{R(q-1)}\longrightarrow s_+,
\qquad
\frac{\Phi(-R)}{R}\longrightarrow -s_-.
\]
Consequently,
\[
\frac{\avg{\Phi(f_R)}_B-\Phi(0)}
     {\|f_R\|_{\ast,B}}
\longrightarrow \frac{s_+-s_-}{2}
=\kappa(\Phi).
\]
Since \(\scC_{\Phi,B}(f_R)\) is at least its Jensen gap at \(B\),
no smaller constant can hold in \eqref{eq:sharp-general-packing}.
\end{proof}
For nonhomogeneous $\Phi$, this is an optimal inequality over all amplitudes, \ie all scalar multiples $\lambda f$, with $\lambda>0$;
the sharp constant cannot be deduced from the case $\lambda=1$ by
scaling. The optimum need not be attained.
For example, $\Phi(t)=\sqrt{1+t^2}$ has $\kappa(\Phi)=1$, but $|\Phi'(t)|<1$ for every finite $t$. Thus, for each $x_j\ne\bar x$, the mean value theorem gives
$|\Phi(x_j)-\Phi(\bar x)|<|x_j-\bar x|$.
Consequently, for every nonconstant finite vector
$(x_1,\ldots,x_q)$,
$0\le
\frac1q\sum_{j=1}^q\Phi(x_j)-\Phi(\bar x)
<
\frac1q\sum_{j=1}^q|x_j-\bar x|$.
In other words, the Jensen gap is strictly smaller than the mean
oscillation at every node where the mean oscillation is positive.

There is a useful representation of this result in terms of threshold
crossings. The right derivative \(\Phi'_+\) is nondecreasing and
right-continuous. Define the nonnegative Lebesgue--Stieltjes measure $\nu$ by
$\nu((a,b])=\Phi'_+(b)-\Phi'_+(a)$. Taking $a\to-\infty$ and $b\to+\infty$ gives
$\nu(\mathbb R)=s_+-s_-$. We claim that
\begin{equation}\label{eq:threshold-representation}
 \Phi(x)=\Phi(0)+\frac{s_++s_-}{2}x
       +\frac12\int_\R\bigl(|x-t|-|t|\bigr)\,\nu(dt).
\end{equation}
To see this, denote the right-hand side by \(G(x)\). The integral is finite since $\bigl||x-t|-|t|\bigr|\le |x|$ and
$\nu(\mathbb R)<\infty$. The right derivative of \(|x-t|\) with respect to \(x\) is \(1\) when
\(t\le x\) and \(-1\) when \(t>x\). Hence
\begin{align*}
G'_+(x)
=\frac{s_++s_-}{2}
  +\frac12\left(\nu((-\infty,x])-\nu((x,\infty))\right)
&=\frac{s_++s_-}{2}
  +\nu((-\infty,x])-\frac{\nu(\mathbb R)}{2}
\\&=s_-+\nu((-\infty,x])
=\Phi'_+(x).
\end{align*}
Both \(G\) and \(\Phi\) are convex, and \(G'_+(x)=\Phi'_+(x)\)
for every \(x\in\mathbb R\). So, they differ by a constant. Since $G(0)=\Phi(0)$,
that constant is zero, proving \eqref{eq:threshold-representation}. Indeed, the affine terms in
\eqref{eq:threshold-representation} cancel in the Jensen defect, and
therefore Applying the representation to the child averages and to their parent average gives
\begin{equation}\label{eq:threshold-defects}
\lambda_I^\Phi(f)
=\frac12\int_{\mathbb R}
       \lambda_I^{|\cdot-t|}(f)\,\nu(dt).
\end{equation}
All integrals here are finite: for every node \(J\),
\(\bigl||f_J-t|-|t|\bigr|\le |f_J|\). Moreover,
\(\lambda_I^{|\cdot-t|}(f)\ge0\) by Jensen's inequality.

For a node \(P\), apply \eqref{eq:threshold-defects} first to a
finite number of generations below \(P\). As the number of
generations increases, the sums of defects increase on both sides.
Monotone convergence then yields
\[
\sum_{I\in\mathcal T(P)}\lambda_I^\Phi(f)\mu(I)
=
\frac12\int_{\mathbb R}
 \left(
   \sum_{I\in\mathcal T(P)}
      \lambda_I^{|\cdot-t|}(f)\mu(I)
 \right)\nu(dt).
\]
By \cref{thm:packing}, the inner sum equals
\[
\mu(P)\left(
  \avg{|f-t|}_P-|f_P-t|
\right)
\le \int_P|f-f_P|\,d\mu,
\]
where the inequality follows from the \(1\)-Lipschitz property of
\(x\mapsto|x-t|\). Thus the packing constant for \(\Phi\) is at most
\(\nu(\mathbb R)/2=(s_+-s_-)/2\).

Sharpness follows from \cref{thm:slope-variation}; it does not require \(\Phi\) to be positively homogeneous.

\subsection{Cancellation and the absolute-value functional}
For this subsection write
\[
 \scA_B(f):=\|\,|f|\,\|_{\ast,B},
 \qquad
 \scA_m(a):=F_{m,q}(|a|).
\]
Set \(f_+=\max\{f,0\}\) and \(f_-=\max\{-f,0\}\).
The exact packing identity (\cref{thm:packing}) gives
\[
 J_P^{|\cdot|}(f)
 =\frac1{\mu(P)}
   \sum_{I\in\cT(P)}\lambda_I^{|\cdot|}(f)\mu(I).
\]
The decomposition $f=f_+-f_-$ yields
\begin{equation}\label{eq:gap-signs}
 D_P(f)=2\min\{\avg{f_+}_P,\avg{f_-}_P\}.
\end{equation}
Indeed, set \(A=\avg{f_+}_P\) and \(B=\avg{f_-}_P\). Since
\(|f|=f_++f_-\) and \(f=f_+-f_-\), the Jensen gap for absolute value is
\[
D_P(f)
=\avg{|f|}_P-|\avg{f}_P|
=A+B-|A-B|
=2\min\{A,B\}.
\]
Thus the defect measures the smaller of the positive and negative
contributions to the average.

\begin{prop}[Amplification requires cancellation]\label{prop:amplification}
For every $f\in\BMO_\cT(B)$,
\begin{equation}\label{eq:A-C}
 \scA_B(f)\le\|f\|_{\ast,B}+\scC_B(f)\le2\|f\|_{\ast,B}.
\end{equation}
At finite depth for nonzero vector $a$, $\scC_m(a)=F_{m,q}(a)>0$ holds if and only if a node
$\sigma$ satisfies $A_\sigma(a)=0$ and $D_\sigma(a)=F_{m,q}(a)$.
\end{prop}
\begin{proof}
Put $M=\avg{|f|}_P$ and $\mu=f_P$. The reverse triangle inequality gives
\[
 \avg{\bigl||f|-M\bigr|}_P
 \le\avg{\bigl||f|-|\mu|\bigr|}_P+M-|\mu|
 \le\osc_P(f)+D_P(f).
\]
Taking suprema proves \eqref{eq:A-C}.
For the equality statement, choose a node attaining $\scC_m(a)$.
Equality in
\[
 \avg{|a|}_\sigma-|A_\sigma|
 \le D_\sigma(a)\le F_{m,q}(a)
\]
forces equality in the triangle inequality at every leaf. If $A_\sigma>0$,
this requires every leaf below $\sigma$ to be at least $A_\sigma$; averaging
then makes the restriction constant. If $A_\sigma<0$, the analogous argument
uses leaves at most $A_\sigma$. Both contradict positive oscillation.
Thus $A_\sigma=0$. The converse follows directly from \eqref{eq:C-finite}.
\end{proof}
Applying \eqref{eq:A-C} to a depth-$m$ function gives
\[
\scA_m(a)-F_{m,q}(a)\le \scC_m(a).
\]
Consequently, if \(F_{m,q}(a)=1\) and
\(\scA_m(a)\ge 1+\eta\), then \(\scC_m(a)\ge\eta\). This is a necessary condition. The converse does not follow: a positive
Jensen defect need not make \(\scA_m(a)>1\).

The one-step defect \(\lambda_I^{|\cdot|}(a)\) has a simple interpretation.
It is zero precisely when the child averages are all nonnegative or
all nonpositive. It is positive precisely when at least one child
average is strictly positive and another is strictly negative. In
the binary case, write the two child averages as \(x+d\) and \(x-d\).
Their parent average is \(x\), and
\[
\lambda_I^{|\cdot|}(a)
=\frac{|x+d|+|x-d|}{2}-|x|
=(|d|-|x|)_+,
\]
\section{Extremal constants for the absolute-value map}\label{sec:absolute}
Here $|\cdot|$ is the real absolute value, distinct from the field absolute
value $|\cdot|_K$. The family $\{\|\, | f |\,\|_{I,\ast}\}$ satisfies \cref{def:scN}. Put
\[
 \scA(f):=\|\,|f|\,\|_\ast,\qquad
 c(q):=c_\scA=\sup_{\|f\|_\ast\le1}\scA(f),\qquad
 \scA_m(a):=F_{m,q}(|a|)=\max_{|\sigma|\le m}D_\sigma(|a|).
\]
Write $c_m(q)=c_{\scA,m}$, $\bar c_m(q)=\bar c_{\scA,m}$, and
$\bar c(q):=\sup_m\bar c_m(q)$.

We refer to \(a\) as the input and \(|a|\) as the output of the
absolute-value map. Thus \(D_\sigma(a)\) and \(D_\sigma(|a|)\)
are their oscillations on node \(\sigma\), while
\(F_{m,q}(a)\) and \(\scA_m(a)=F_{m,q}(|a|)\) are their full tree
BMO seminorms.
\subsection{Reduction, anchoring, and exact finite optimization}
\label{subsec:absolute-LP}
\begin{prop}\label{prop:absolute-structure}
The absolute-value functional satisfies \cref{thm:reduction}. Its finite-depth functionals $\scA_m$ are continuous, rational piecewise linear, zero-leaf anchored,
and hereditary. Therefore
\begin{equation}\label{eq:absolute-reduction}
 c(q)=\sup_m c_m(q)=\sup_m\bar c_m(q),
 \qquad c_m(q)\le\bar c_{m+\lceil qm/(q-1)\rceil}(q).
\end{equation}
For each fixed $m$, both finite constants are rational and attained by
rational vectors.
\end{prop}
\begin{proof}
Homogeneity, vanishing on constants, invariance under tree isomorphisms, and heredity follow
from the node formulas. If $E_m^If\to f$ in $\operatorname{L^1}(I)$, then
$|E_m^If|\to|f|$ in $\operatorname{L^1}(I)$. Equation \eqref{eq:osc-continuity} gives
convergence on every fixed descendant. Taking a lower limit after choosing a
node near the supremum proves \eqref{eq:LSC}.

To prove zero-leaf anchoring, fix $a\in V_{m,q}$ and consider $h(c)=F_{m,q}(|a+c\one|),\, c\in\R$.
The finitely many values \(-a_\omega\) divide \(\R\) into
intervals. On each interval, the sign of every coordinate
\(a_\omega+c\) is fixed. So
$|a+c\one|$ is affine in $c$, and \(h\) is convex there because \(F_{m,q}\)
is a seminorm. On a bounded interval, \(h\) therefore attains
its maximum at an endpoint. If \(c\) lies outside the smallest and largest values of
\(-a_\omega\), all coordinates of \(a+c\one\) have the same sign.
Translation invariance and symmetry of \(F_{m,q}\) then give
\(h(c)=F_{m,q}(a)\). This value occurs at the endpoint of the
corresponding outer interval as well. Hence the maximum over
\(c\in\R\) occurs at \(c=-a_\omega\) for some leaf \(\omega\). Thus a maximizing translation occurs at a zero leaf.

Each \(D_\sigma(|a|)\) is obtained from finitely many absolute values
and averages with rational coefficients. Hence
\(\scA_m(a)=\max_\sigma D_\sigma(|a|)\) is rational piecewise linear. Apply \cref{thm:finite-anchor,thm:stable-centering}.
For the centered constant, since \(\overline\Sigma_{m,q}\) is a compact rational polytope and
\(\scA_m\) is rational piecewise linear, maximizing on its finitely
many linearity regions shows that \(\bar c_m(q)\) is rational and is
attained by a rational vector.
\end{proof}
No conclusion about rationality or attainment of $c(q)$ follows from the
fixed-depth assertion.
\begin{remark}\label{exact-family-finite-linear-programs}
There is a useful simplification of the finite optimization. Let
\[
R_m(q):=\sup_{F_{m,q}(a)\le1}
D_\emptyset(|a|)=\sup_{F_{m,q}(a)\le1}
\frac1{q^m}\sum_{\omega\in\Omega_{m,q}}
\left|
  |a_\omega|
  -\frac1{q^m}\sum_{\omega'\in\Omega_{m,q}}|a_{\omega'}|\right|.
\]
Then $R_m(q)=c_m(q)$. Leaves contribute zero oscillation, so it suffices to consider nonleaf nodes, whose remaining depth is at least \(1\). Indeed, the definition gives $R_d(q)\le R_m(q)$
for $d\le m$, and any node of an admissible depth-$m$ vector carries an
admissible vector of the smaller remaining depth. Its absolute-value
oscillation is at most $R_d(q)\le R_m(q)$. Conversely, \(\scA_m(a)=F_{m,q}(|a|)\) is at least the oscillation of
\(|a|\) at the root. Taking the supremum over all \(a\) with
\(F_{m,q}(a)\le1\) gives \(c_m(q)\ge R_m(q)\), and hence
\(c_m(q)=R_m(q)\).

The identity \(c_m(q)=R_m(q)\) reduces the optimization to the
oscillation of \(|a|\) at the root. Writing $n=|\Omega_{m,q}|=q^m$, we show that $c_m(q)$ is the maximum of the following finite
family of rational linear programs, indexed by
$s\in\{-1,1\}^{\Omega_{m,q}}$ and $S\subseteq\Omega_{m,q}$:
\begin{equation}\label{eq:abs-LP}
 \begin{aligned}
 \text{maximize }&\quad
  \frac2n\left(\sum_{\omega\in S}s_\omega a_\omega
       -\frac{|S|}{n}\sum_{\omega\in\Omega_{m,q}} s_\omega a_\omega\right),\\
 \text{subject to }&\quad F_{m,q}(a)\le1,\quad
             s_\omega a_\omega\ge0\ (\omega\in\Omega_{m,q}),\quad
             a_{\omega_0}=0.
 \end{aligned}
\end{equation}
Here $\omega_0$ is any fixed leaf. By \cref{lem:compact},
\(F_{m,q}(a)\le1\) is equivalent to finitely many rational linear
inequalities. To verify the objective, set \(y_\omega=|a_\omega|\) and
\(\bar y=n^{-1}\sum_{\omega\in\Omega_{m,q}}y_\omega\). Since
\(\sum_\omega(y_\omega-\bar y)=0\), choosing \(S\) to contain the
nonnegative deviations gives
\[
 D_\emptyset(|a|)
=\frac1n\sum_{\omega\in\Omega_{m,q}}|y_\omega-\bar y|
 =\max_{S\subseteq\Omega_{m,q}}
    \frac2n\sum_{\omega\in S}(y_\omega-\bar y).
\]
For a fixed sign vector \(s\), the constraints
\(s_\omega a_\omega\ge0\) give \(y_\omega=s_\omega a_\omega\);
thus each choice of \(s\) and \(S\) yields the linear objective
in \eqref{eq:abs-LP}.

It remains to justify the constraint \(a_{\omega_0}=0\). This follows from a convexity argument similar to the one above. Indeed,
the function \(c\mapsto D_\emptyset(|a+c\one|)\) is convex between
consecutive values of \(-a_\omega\) and constant outside their range.
Hence its maximum is attained at \(c=-a_\omega\) for some leaf
\(\omega\). This translation preserves \(F_{m,q}(a)\) and makes
the value at \(\omega\) zero. A rooted tree automorphism can move
\(\omega\) to the prescribed leaf \(\omega_0\) while preserving both
\(F_{m,q}\) and the root objective. For the resulting vector $b$, choose its signs \(s=(s_\omega)\) so that $s_\omega b_\omega=|b_\omega|$ and set \[
 S=\{\omega\in\Omega_{m,q}:
             |b_\omega|\ge A_\emptyset(|b|)\}.
\]
Then $b$ is feasible for the corresponding program, whose objective equals \(D_\emptyset(|b|)\). Conversely, the objective of every feasible program is at most \(D_\emptyset(|a|)\), and hence at most \(R_m(q)\). Thus the family of linear programs in \eqref{eq:abs-LP} is exact, \textit{i.e.}, it computes $R_m(q)=c_m(q)$.
\end{remark}

\subsection{A sharp scalar inequality and the depth-one constant}
For an integrable real random variable $X$, write $\E$ for expectation and set
\[
\mu:=\E X,\qquad
X_+:=\max\{X,0\},\qquad X_-:=\max\{-X,0\},\qquad d(X):=\E|X-\E X|.
\]
For an event $A$, let $\chi_A$ denote its indicator, and write $A^c:=\Omega\setminus A$. The next lemma is useful beyond finite trees.
\begin{lemma}[A tail identity for folding]\label{lem:folding-tail}
After replacing $X$ by $-X$ if necessary, let $\mu=\E X\ge0$ and
$t=\E|X|$. Set $S=\{|X|>t\}$ and $p=\Prob(S)$. Define
\[
 N=\E(X_-\chi_S),\quad R=\E(X_-\chi_{S^c}),\quad
 D=\E((\mu-X)_+\chi_{S^c}),\quad
 U=\E((X-\mu)_+\chi_{S^c}),
\]
and $r=\Prob(S\cap\{X<0\})$. Then
\begin{align}
 \frac12d(X)&=N+r\mu+D,\label{eq:tail-input}\\
 \frac12d(|X|)&=2(1-p)N-2pR+D-U.\label{eq:tail-output}
\end{align}
Consequently,
\begin{equation}\label{eq:tail-bound}
 d(|X|)\le\max\{1,2(1-p)\}\,d(X).
\end{equation}
If $S$ has only one sign, then $d(|X|)\le d(X)$.
\end{lemma}
\begin{proof}
For any integrable real random variable \(Y\),
\[
\frac12d(Y)
=\E(\E Y-Y)_+
=\E(Y-\E Y)_+.
\]
Since \(t=\E|X|\ge\mu\), every positive value of \(X\) in \(S\)
exceeds \(\mu\). Thus \((\mu-X)_+=0\) on
\(S\cap\{X>0\}\). On \(S\cap\{X<0\}\), we have
\((\mu-X)_+=\mu+X_-\). The contribution from \(S^c\) is \(D\), so
\[
\frac12d(X)
=\E(\mu-X)_+
=N+r\mu+D,
\]
which proves \eqref{eq:tail-input}.

Set \(H=\E(X_+\chi_S)\). Because
\(\E X=\E X_+-\E X_-\), we have
$t=\mu+2(N+R)$.
Splitting the expectation of \(X\) over \(S\) and \(S^c\) gives
$H-N+\E(X\chi_{S^c})=\mu$.
Moreover, the identity
\((\mu-X)_+-(X-\mu)_+=\mu-X\) yields
$D-U=(1-p)\mu-\E(X\chi_{S^c})$. Since \(S=\{|X|>t\}\),
\begin{align*}
\frac12d(|X|)
 =\E(|X|-t)_+
 =H+N-pt
 =2(1-p)N-2pR+D-U.
\end{align*}
This proves \eqref{eq:tail-output}.

Now \(R,U,r,\mu\ge0\). Dropping the nonpositive terms in
\eqref{eq:tail-output} and using \eqref{eq:tail-input}, we obtain
\begin{align*}
\frac12d(|X|)
 \le 2(1-p)N+D
 \le \max\{1,2(1-p)\}(N+D)
 \le \frac12\max\{1,2(1-p)\}d(X).
\end{align*}
This is \eqref{eq:tail-bound}. It remains to prove the stronger conclusion when \(S\) has only one
sign. If \(X\ge0\) on \(S\), then \(t\ge\mu\) gives
\[
\frac12d(|X|)
 =\E((X-t)\chi_S)
 \le\E((X-\mu)\chi_S)
 \le\E(X-\mu)_+
 =\frac12d(X).
\]
If \(X\le0\) on \(S\), then \(t\ge-\mu\) gives
\[
\frac12d(|X|)
 =\E((-X-t)\chi_S)
 \le\E((\mu-X)\chi_S)
 \le\E(\mu-X)_+
 =\frac12d(X).
\]
If \(S\) is empty, then \(|X|\le t=\E|X|\) almost surely, so
\(|X|=t\) almost surely and \(d(|X|)=0\).
\end{proof}
\begin{thm}[Sharp folding on equal atoms]\label{thm:scalar-sharp}
For $n\ge2$ and $a_1,\ldots,a_n\in\R$,
\begin{equation}\label{eq:scalar-sharp}
 \frac1n\sum_{i=1}^n\bigl||a_i|-\overline{|a|}\bigr|
 \le\delta_n\frac1n\sum_{i=1}^n|a_i-\bar a|,
 \qquad \delta_n:=\max\{1,2-4/n\}.
\end{equation}
The constant is sharp.
\end{thm}
\begin{proof}
Apply \cref{lem:folding-tail} on $n$ equal atoms. If the large-magnitude set
$S$ has only one sign, contraction holds. Otherwise it contains at least two
atoms, so $p\ge2/n$ and \eqref{eq:tail-bound} gives \eqref{eq:scalar-sharp}.
A nonconstant nonnegative vector gives ratio $1$. The vector
$(-n/2,n/2,0,\ldots,0)$ has input oscillation $1$ and output oscillation
$2-4/n$. These examples prove sharpness for every $n$.
\end{proof}

\begin{cor}[Finite-depth bounds]\label{cor:finite-abs}
For all $q\ge2$ and $m\ge1$,
\[
 c_1(q)=\delta_q,\qquad
 c_m(q)\le\delta_{q^m}<2.
\]
In particular, $c_1(2)=c_2(2)=1$.
\end{cor}
\begin{proof}
At depth one the input and output have just one nonzero node oscillation,
so \cref{thm:scalar-sharp} is exactly the extremal problem.
At an arbitrary node of a depth-$m$ tree, apply the theorem to its
$q^{m-|\sigma|}$ equal leaves. The constants $\delta_n$ are nondecreasing in $n$.
For $q=2,m=2$, $\delta_4=1$, and a nonnegative nonconstant vector gives
the reverse bound.
\end{proof}
The same argument applies to a finite probability space in which
every atom has probability at least \(\beta\). Let
\(S=\{|X|>\E|X|\}\). If \(S\) contains values of both signs, it
contains at least two atoms, so \(\Prob(S)\ge 2\beta\). The bound
in \eqref{eq:tail-bound} then gives
\[
 d(|X|)\le \max\{1,2-4\beta\}\,d(X).
\]
If \(S\) contains values of only one sign, \cref{lem:folding-tail}
gives the stronger bound \(d(|X|)\le d(X)\). Hence
\(\max\{1,2-4\beta\}\) is an upper bound for the folding constant
on this probability space. This is a bound for one averaging set;
it does not give a uniform tree-BMO bound unless the required
lower bound on relative atom probabilities holds at every node
under consideration.

\subsection{Centered constructions beyond two generations}
Set $\Delta_q:=2-4/q+4/q^2$. The following centered
depth-two vector gives the lower bound \(\Delta_q\) for the
absolute-value map. On two root children of the root, assign the leaf values
\[
 (q,0,q/2,\ldots,q/2),\qquad
 (-q,0,-q/2,\ldots,-q/2),
\]
and put zero on all remaining root children. The root mean is zero, its mean
absolute value is $1$, and each nonconstant child has oscillation $1$.
Thus $F_{2,q}=1$. There are $q^2-2q+2$ zero leaves, and all nonzero absolute
values are at least $1$. The root absolute-value oscillation is therefore
$2(q^2-2q+2)/q^2=\Delta_q$. Child absolute-value oscillations are $1$.
Hence
\begin{equation}\label{eq:depth-two-lower}
 \bar c_2(q)\ge\Delta_q.
\end{equation}

\begin{prop}[A stronger centered $q$-ary construction]\label{thm:q-counterexample}
For every $q\ge3$ there is a centered depth-three vector with
\begin{equation}\label{eq:q-counterexample}
 F_{3,q}(a)=1,\qquad
 \scA_3(a)=\Delta_q+\frac{q-2}{q^3}.
\end{equation}
Consequently $c(q)=\sup_m\bar c_m(q)>\Delta_q$ for every $q\ge3$. In particular, \(\Delta_q\) is not the global sharp constant for the absolute-value map.
\end{prop}
\begin{proof}
Put
\[
 u:=\frac{q}{2(q-1)},\qquad v:=qu=\frac{q^2}{2(q-1)}.
\]
On each of $q-1$ root children, select one grandchild and give its $q$ leaves
the values
\begin{equation}\label{eq:positive-block}
 v+q/2,\quad v-q/2=u,\quad v,\ldots,v.
\end{equation}
Fill the other $q-1$ grandchildren with zero. The selected grandchild has
mean $v$ and oscillation $1$. All its values are at least $u$.
The root child has mean $u$ and oscillation
\[
 \frac1q(v-u)+\frac{q-1}{q}u
       =\frac{2(q-1)}q u=1.
\]
It is nonnegative, so its absolute-value oscillations are the same.
On the remaining root child put, on its $q$ grandchildren, the constant
values
\begin{equation}\label{eq:negative-block}
 -q,\quad0,\quad-q/2,\ldots,-q/2,
\end{equation}
and refine them constantly to depth three. This child has mean $-q/2$ and
oscillation $1$. Since $(q-1)u=q/2$, the root mean is zero and the root mean
absolute value is $1$. Every other nonzero oscillation has already been
checked, proving $F_{3,q}(a)=1$.

There are $2q(q-1)$ nonzero leaves among $q^3$ leaves. Of these, exactly
$q-1$ have value $u<1$; all remaining absolute values are at least $1$.
For a nonnegative variable of mean $1$, its oscillation is twice its mean
deficit below $1$. Thus the root output oscillation is
\begin{align*}
 2\left(1-\frac{2(q-1)}{q^2}
           +\frac{q-1}{q^3}(1-u)\right)
 &=2-\frac4q+\frac4{q^2}+\frac{q-2}{q^3}.
\end{align*}
All proper descendants have output oscillation at most $1$, so this is the
full output seminorm. The vector is centered and the asserted strict
inequality follows.
\end{proof}
For example, the new lower bounds are $31/27$ for $q=3$ and $41/32$ for
$q=4$. For $q=3$, the three nine-leaf blocks can be written as
\[
 (3/4,9/4,15/4,0,0,0,0,0,0),\quad
 (3/4,9/4,15/4,0,0,0,0,0,0),
\]
\[
 (-3,-3,-3,0,0,0,-3/2,-3/2,-3/2).
\]
Their root mean is zero and their full input seminorm is $1$.

\subsection{Failure of contractivity on binary trees}
By a strong (respectively, rooted) binary tree, we mean a strong (respectively, rooted) homogeneous $\frac12$-tree. For \(q=2\), the absolute-value map is contractive at depths one
and two. The next vector shows that contractivity fails at
depth four: its output BMO seminorm exceeds its input BMO
seminorm, \text{i.e.}, the global sharp constant $c(2)$ for absolute-value map cannot be $1$.
\begin{thm}[An explicit binary counterexample]\label{thm:binary-counterexample}
On the depth-four binary tree, order leaves lexicographically and set
\begin{equation}\label{eq:binary-vector}
 a=(0,-8,-12,-12,\ 0,0,0,0,\ 6,1,1,1,\ 1,10,14,14).
\end{equation}
Then
\[
 F_{4,2}(a)=5,\qquad \scA_4(a)=\frac{41}{8}.
\]
There is a centered depth-five extension with the same two seminorms.
In particular,
\begin{equation}\label{eq:binary-lower}
 c_4(2)\ge\frac{41}{40},\qquad
 \bar c_5(2)\ge\frac{41}{40},\qquad c(2)\ge\frac{41}{40}>1.
\end{equation}
\end{thm}
\begin{proof}
The root mean is $1$. The left and right children of the root have means \(-4\)
and \(6\), respectively. Direct averaging gives the input oscillations in
\cref{tab:binary}; entries in each row are ordered from left to right.
\begin{table}[ht]
\centering
\caption{Exact node data for \eqref{eq:binary-vector}. Leaves have zero oscillation.}
\label{tab:binary}
\begin{tabular}{cll}
\toprule
Depth & Node means & Input oscillations\\
\midrule
0 & $1$ & $5$\\
1 & $-4,\ 6$ & $5,\ 5$\\
2 & $-8,\ 0,\ 9/4,\ 39/4$ & $4,\ 0,\ 15/8,\ 35/8$\\
3 & $-4,\ -12,\ 0,\ 0,\ 7/2,\ 1,\ 11/2,\ 14$
  & $4,\ 0,\ 0,\ 0,\ 5/2,\ 0,\ 9/2,\ 0$\\
\bottomrule
\end{tabular}
\end{table}
This proves $F_{4,2}(a)=5$. Both root children have one sign, so every proper
node has identical input and output oscillations. At the root, the mean
absolute value is $5$. The magnitudes below $5$ consist of five zeros and
four ones. Their total deficit is $5\cdot5+4\cdot4=41$. Hence
\[
 D_\emptyset(|a|)=\frac2{16}\,41=\frac{41}{8},
\]
which exceeds every proper-node oscillation.

To center the vector, place \(a\) on one child of a new root
and assign the constant value \(-1\) to the other child,
refining that child to sixteen leaves. The new root has mean
zero and input oscillation
\[
 \frac12\avg{|a|}+\frac12|-1|=\frac{5+1}{2}=3.
\]
Its mean absolute value is \(3\). Since
\(\sum_{i=1}^{16}\bigl||a_i|-3\bigr|=78\), its output
oscillation is
\[
 \frac12\cdot\frac{78}{16}
 +\frac12|1-3|=\frac{55}{16}.
\]
Both new-root oscillations are smaller than the corresponding
seminorms of the retained subtree. Thus the depth-five vector
has input seminorm \(5\) and output seminorm \(41/8\).
Dividing its values by \(5\) gives the centered admissible
vector in \eqref{eq:binary-lower}.

\end{proof}
This one-step extension follows the construction in
\cref{thm:stable-centering}: after dividing \(a\) by \(5\),
its mean is \(1/5\), and the constant \(-1/5\) on the other
child makes the new root mean zero. The original depth-four
vector remains unchanged inside one child, where its output
seminorm exceeds its input seminorm. Therefore, an estimate
at the newly added root alone cannot establish contractivity
for the entire centered tree.

\subsection{A recursive binary construction}
The preceding depth-four vector gives a counterexample
to binary contractivity. The next construction improves its
lower bound. It uses a measurable subset $E$ of a root $B$ in a binary rooted tree whose indicator has unusually small BMO seminorm, then approximates the resulting function by finite-depth vectors.

\begin{lemma}[The least nonzero indicator seminorm]\label{lem:indicator}
On a rooted homogeneous $\dfrac1q$-tree $\cT(B)$,
\begin{equation}\label{eq:indicator-minimum}
 \inf_{\substack{E\subset B\ \mathrm{measurable}\\
                  0<\mu(E)<\mu(B)}}\|\chi_E\|_{\ast,B}
       =\frac{2q}{(q+1)^2},
\end{equation}
and the infimum is attained by a measurable set \(E\subset B\).
\end{lemma}
\begin{proof}
Fix a measurable $E\subset B$ with $0<\mu(E)<\mu(B)$.
Put $a=1/(q+1)$ and $p_I=\mu(E\cap I)/\mu(I)$.
The oscillation of $\chi_E$ on $I$ is $2p_I(1-p_I)$.
If $p_B\in[a,1-a]$, the lower bound follows at the root. Replacing $E$
by its complement $B\setminus E$ if needed, the remaining case is $0<p_B<a$.
At almost every point of $E$, the conditional densities tend to $1$.
Along such a path, consider the first node $J$ whose density is at least $a$.
Its parent has density less than $a$, so $p_J\le q p_{\operatorname{parent}(J)}
<qa=1-a$. Thus some node has density in $[a,1-a]$, proving the lower bound.

To attain it, prescribe root density $a$. A node of density $a$ has one
child of density $1-a$ and $q-1$ children of density $0$. A node of density
$1-a$ has one child of density $a$ and $q-1$ children of density $1$.
Freeze the zero and one nodes. There is one unfrozen node at every generation,
of relative measure $q^{-k}$, so the union of the nodes frozen at density $1$
defines a measurable set with exactly the prescribed conditional densities.
The only nonzero indicator oscillation is $2a(1-a)=2q/(q+1)^2$.
\end{proof}
The extremizing set is constructed as a union of descendant
nodes; it is generally not a single node.
\begin{thm}\label{thm:binary-recursive}
On every homogeneous binary rooted tree there is a bounded function $f$ with
\[
 \|f\|_{\ast,B}=1,\qquad \|\,|f|\,\|_{\ast,B}=\frac{25}{24}.
\]
The same ratio is attained by a centered bounded function. In particular,
$c(2)=\bar c(2)\ge25/24$.
\end{thm}
\begin{proof}
Let $B_-,B_+$ be the children of a root $B$. Inside the child \(B_-\), use the measurable set \(E\) supplied by
\cref{lem:indicator} for \(q=2\). Relative to \(B_-\), this set has measure \(1/3\); its densities on descendant nodes
belong to \(\{0,1/3,2/3,1\}\), and
\begin{equation}\label{eq:third-indicator}
 \|\chi_E\|_{\ast,B_-}=\frac49.
\end{equation}
On $B_-$ set
$f=-\frac94\chi_E$, using the relative construction above. On the two
children of $B_+$ set $f=1/4$ and $f=9/4$, respectively. Then
\[
 f_{B_-}=-\frac34,\qquad f_{B_+}=\frac54,\qquad
 f_B=\frac14,\qquad \avg{|f|}_B=1.
\]
Every value on $B_-$ is at most $f_B$, and every value on $B_+$ is at least
$f_B$. Consequently $\osc_B(f)=1$. Equation \eqref{eq:third-indicator}
gives $\|f\|_{\ast,B_-}=1$, and the two-valued right child has oscillation $1$.
Thus $\|f\|_{\ast,B}=1$.

At the root with probabilities computed using the normalized measure
\(\mu/\mu(B)\), the absolute values of \(f\) satisfy
\[
 \Prob\{|f|=0\}=\frac13,\qquad
 \Prob\{|f|=\tfrac14\}=\frac14,\qquad
 \Prob\{|f|=\tfrac94\}=\frac5{12}.
\]
Since its mean is $1$, its mean absolute deviation is
\[
 \osc_B(|f|)=2\left(\frac13+\frac14\cdot\frac34\right)
            =\frac{25}{24}.
\]
On each proper descendant, $f$ has one sign, so its input and output
oscillations agree and are at most $1$. This proves the first assertion.

For a centered version, place $f$ and $-f$ on the two children of another
root. The new mean is zero and the new input root oscillation is
$\avg{|f|}_B=1$. The output distribution at the new root is unchanged.
All descendant input oscillations remain at most $1$. Thus \(g\) has the asserted centered ratio. Norm-preserving extension (\cref{lem:extension}) gives the global examples.
\end{proof}

For completeness, we give finite-depth approximations of this
construction. Use the dyadic interval model with root
\(B=[0,1)\) and left child \(B_-=[0,1/2)\). In the proof of
\cref{thm:binary-recursive}, choose
\[
 E=[0,1/6)\subset B_-.
\]
Under the normalized coordinate \(u=2x\) on \(B_-\), this is
the set \([0,1/3)\), which attains the bound in
\cref{lem:indicator} for \(q=2\). For $m\ge2$ put
$n=2^{m-1}$, $k=\lfloor n/3\rfloor$, and $r=n-3k\in\{1,2\}$. In the formulas below, we identify \(f_m\) with its ordered
vector of depth-\(m\) leaf values. The depth-$m$ conditional expectation $f_m$ has, on the left half, the values
\[
 \underbrace{-9/4,\ldots,-9/4}_{k\text{ leaves}},\quad
 -3r/4,\quad
 \underbrace{0,\ldots,0}_{n-k-1\text{ leaves}},
\]
and, on the right half, $n/2$ values $1/4$ followed by $n/2$ values $9/4$.
The order is part of the construction. Contractive approximation and the
root calculation give $F_{m,2}(f_m)=1$. Its root output oscillation is
\begin{equation}\label{eq:binary-finite-family}
 D_\emptyset(|f_m|)=
 \begin{cases}
  25/24-5/(12n),&r=1,\\
  25/24-1/(3n),&r=2.
 \end{cases}
\end{equation}
Indeed, the root mean of $|f_m|$ is $1$; the zero leaves contribute
$(n-k-1)/(2n)$ to the negative deviation, the right-hand $1/4$ leaves
contribute $3/16$, and the exceptional left leaf contributes $1/(8n)$
exactly when $r=1$. Doubling this sum proves
\eqref{eq:binary-finite-family}. All proper-node output oscillations are
at most $1$, and the right child has oscillation $1$. Therefore
\[
 \scA_m(f_m)=\max\{1,D_\emptyset(|f_m|)\}\longrightarrow\frac{25}{24}.
\]
Placing \(f_m\) and \(-f_m\) on the two children of a new
root produces a centered depth-\((m+1)\) vector. Its input
seminorm is \(1\), and its output seminorm is
\(\scA_m(f_m)\): the new root has the same distribution of
absolute values as either child. Hence these centered finite
vectors have ratios tending to \(25/24\).

\subsection{An all-depth upper bound from \texorpdfstring{$\operatorname{L}^2$}{L2}-BMO}
The constructions above give lower bounds. We now combine a
scalar variance estimate with an \(\operatorname{L}^1\)-to-\(\operatorname{L}^2\) BMO
comparison to obtain an upper bound at every depth. We first
apply the comparison to finite-depth functions, which are
bounded, and then use the finite-tree reduction.

Here, we use the notions of \cref{lem:folding-tail}.
\begin{lemma}[A variance improvement]\label{lem:variance-folding}
Let $K\ge1$ and $\Var(X)=\E\bigl[(X-\E X)^2\bigr]$. If $X$ is square integrable and
$d(X)\le1$, $\Var(X)\le K^2$, then
\begin{equation}\label{eq:variance-folding}
 d(|X|)\le\frac{2K^2}{K^2+1}.
\end{equation}
\end{lemma}
\begin{proof}
Use the notation of \cref{lem:folding-tail} and put $H=d(|X|)/2$.
If $p\ge1/2$, \eqref{eq:tail-bound} gives $H\le1/2$.
If $p\le1/2$, the same inequality gives $H\le1-p$.
For $Y=|X|$, we have
\[
 H=\E((Y-\E Y)\chi_S)
   =\E((Y-\E Y)(\chi_S-p)).
\]
The mean minimizes squared deviation, and the reverse triangle
inequality gives
\[
 \Var(|X|)
 \le \E\bigl(|X|-|\E X|\bigr)^2
 \le \E|X-\E X|^2
 =\Var(X).
\]
Cauchy--Schwarz therefore yields
\(H\le K\sqrt{p(1-p)}\).
On $[0,1/2]$, the function $K\sqrt{p(1-p)}$ increases and $1-p$ decreases.
They meet at $p=1/(K^2+1)$, where their common value is
$K^2/(K^2+1)\ge1/2$. Combining the two bounds proves
\eqref{eq:variance-folding}.
\end{proof}

\begin{thm}[Uniform strict improvement of the factor $2$]\label{thm:global-upper}
For every homogeneous $q$-ary tree,
\begin{equation}\label{eq:global-upper}
 c(q)\le U_q:=\frac{2(q+1)^2}{q^2+3q+1}
             =2-\frac{2q}{q^2+3q+1}<2.
\end{equation}
\end{thm}
\begin{proof}
For a rooted tree write
\[
 \|f\|_{\ast,2;B}:=\sup_{P\in\cT(B)}
             \bigl(\avg{|f-f_P|^2}_P\bigr)^{1/2}.
\]
Slavin and Vasyunin's comparison on an $\alpha$-tree gives
\begin{equation}\label{eq:SV-comparison}
 \|f\|_{\ast,2;B}\le\frac{1+\alpha}{\sqrt\alpha}\|f\|_{\ast,B}
 \qquad(f\in\BMO^2_\cT(B));
\end{equation}
see \cite[Theorem~1.8, consequence~(1.8)]{SlavinVasyunin2016AlphaTrees}.
A homogeneous $q$-ary rooted tree is an $\alpha$-tree with $\alpha=1/q$.
For a bounded finite-depth function with $\|f\|_{\ast,B}\le1$, this gives
$\|f\|_{\ast,2;B}\le K_q:=(q+1)/\sqrt q$.
On each descendant $P$, regard $f$ as a random variable for normalized
measure. It has $d(f)\le1$ and variance at most $K_q^2$.
\Cref{lem:variance-folding} bounds its output oscillation by
$2K_q^2/(K_q^2+1)=U_q$.
Thus $c_m(q)\le U_q$ at every finite depth $m$.
Since
\(c(q)=\sup_m c_m(q)\) by \eqref{eq:absolute-reduction},
the same bound holds for every \(\operatorname{L}^1\)-BMO input. This passage
does not require an arbitrary \(\operatorname{L}^1\)-BMO input to belong to
\(\operatorname{L}^2\).
\end{proof}
Together with \cref{thm:q-counterexample,thm:binary-recursive}, this proves
\eqref{eq:intro-binary}--\eqref{eq:intro-general}. More precisely,
\[
 c_m(q)\le\min\{\delta_{q^m},U_q\}.
\]
We do not assert that $U_q$ is sharp. The exact global constant remains between the explicit lower
bounds found in \cref{thm:q-counterexample,thm:binary-recursive} and \(U_q\).

\subsection{A centered root counterexample and the scope of induction}
There is a direct way to put the binary amplification at a centered root,
not merely inside a centered tree. Let $a$ be \eqref{eq:binary-vector} and
make a depth-five vector $b=(a,-a)$, preserving the internal order in each
half. Then $A_\emptyset(b)=0$ and
\[
 D_\emptyset(b)=\frac1{16}\sum_i|a_i|=5.
\]
All proper-node input oscillations are at most $5$, so $F_{5,2}(b)=5$.
The two halves have the same magnitude distribution, and therefore
\[
 D_\emptyset(|b|)=\frac{41}{8},\qquad
 \scA_5(b)=\frac{41}{8}.
\]
Hence centering at the root does not restore the inequality
\(D_\emptyset(|b|)\le F_{5,2}(b)\). The two children of the
root have averages \(1\) and \(-1\). Within either child, its
two descendant subtrees have values of fixed sign, yet the different absolute values and how often
they occur give 
\(D_\emptyset(|b|)=41/8>5\). An argument based only on the
averages of the root children therefore misses the
oscillation responsible for amplification.

\section{A sharp norm for local martingale-difference envelopes}
\label{sec:martingale}
Differences between the averages of a node and its children
provide another family of local functionals. Their sharp
constants give explicit examples of our finite-tree
formalism. They also distinguish two rationality phenomena:
an envelope involving an \(\operatorname{L}^p\) average can have an
irrational constant already at finite depth, while a
weighted envelope can have rational constants at every
finite depth and an irrational global constant.

\subsection{A one-generation martingale-difference envelope}
The difference between a node average and its child averages
measures the change across one generation. We take the
\(\operatorname{L}^p\) size of this difference on each node and then the
supremum over nodes. This gives a family of local functionals
whose sharp constants can be computed directly.

For $I$ with children $I^{(1)},\ldots,I^{(q)}$, define
\[
 \Delta_I f:=\sum_{j=1}^q(f_{I^{(j)}}-f_I)\chi_{I^{(j)}},\qquad
 \Gamma_p(I;f):=\left(\avg{|\Delta_I f|^p}_I\right)^{1/p}
 \quad(1\le p<\infty),
\]
and $\Gamma_\infty(I;f):=\max_j|f_{I^{(j)}}-f_I|$. Let
\[
 \mathcal G_{p,B}(f):=\sup_{I\in\cT(B)}\Gamma_p(I;f),\qquad
 \mathcal G_p(f):=\sup_{I\in\cT}\Gamma_p(I;f).
\]
These are local and global martingale-difference envelopes.

\begin{prop}[Sharp one-generation envelope]\label{thm:gradient}
For $1\le p\le\infty$,
\begin{equation}\label{eq:gradient-sharp}
 \mathcal G_{p,B}(f)\le
 \left(\frac q2\right)^{1-1/p}\|f\|_{\ast,B},
\end{equation}
where $1/\infty=0$. The constant is sharp locally, globally, and at every
finite depth $m\ge1$. Centered depth-one functions attain it.
\end{prop}
\begin{proof}
Put $d_j=f_{I^{(j)}}-f_I$. Then $\sum_jd_j=0$ and
\[
 \frac1q\sum_j|d_j|\le\osc_I(f).
\]
If $T=\sum_{d_j>0}d_j=-\sum_{d_j<0}d_j$, convexity gives
$\sum_j|d_j|^p\le2T^p$ for $1\le p<\infty$. Since
$T\le q\osc_I(f)/2$, this proves the desired inequality on $I$.
For $p=\infty$, each $|d_j|\le T$ gives the same endpoint.
Taking the supremum proves \eqref{eq:gradient-sharp}.

Assign the child values $(q/2,-q/2,0,\ldots,0)$ to a root and make each
child constant. The root mean is zero, its mean oscillation is $1$, and
\[
 \Gamma_p=\left(\frac2q\left(\frac q2\right)^p\right)^{1/p}
           =\left(\frac q2\right)^{1-1/p}.
\]
For $p=\infty$ the value is $q/2$. Constant refinement and the extension
lemma give all the claimed sharpness statements.
\end{proof}
The functionals \(\mathcal G_{p,B}\) satisfy the hypotheses of
\cref{thm:reduction}. They are homogeneous, translation
invariant, hereditary, and invariant under rooted tree
isomorphisms. For every fixed node, the child averages of
\(E_m^Bf\) agree with those of \(f\) once \(m\) is sufficiently
large, while \(E_m^Bf\) is constant
on nodes at or below depth \(m\). Consequently,
\[\mathcal G_{p,B}(E_m^Bf) =\max_{\substack{I\in\cT(B)\\
                  \operatorname{depth}_B(I)<m}}
      \Gamma_p(I;f)\longrightarrow\mathcal G_{p,B}(f),\] which
implies \eqref{eq:LSC}.

Subtracting the root average centers a finite vector without
changing its input seminorm or its envelope value. For
\(p=1\) and \(p=\infty\), the finite-depth functionals $\mathcal G_{p,m}$ are rational
piecewise linear. For intermediate \(p\), they need not be:
when \(q=3\) and \(p=2\), the sharp constant at every positive
depth is \(\sqrt{3/2}\).

\subsection{Cumulative martingale differences}
The preceding envelope measures changes across one
generation. For an integer \(N\ge1\), define the
\(N\)-generation difference envelope by
\[
 \Gamma_{p,N}(I;f):=\left(\avg{|E_N^I f-f_I|^p}_I\right)^{1/p},
 \qquad
 \mathcal H_{p,N;B}(f):=\sup_{I\in\cT(B)}\Gamma_{p,N}(I;f),
\]
with the usual maximum when $p=\infty$.

\begin{thm}[Sharp endpoints for cumulative differences]\label{thm:cumulative}
For every $N\ge1$,
\begin{equation}\label{eq:cumulative-endpoints}
 \mathcal H_{1,N;B}(f)\le\|f\|_{\ast,B},\qquad
 \mathcal H_{\infty,N;B}(f)\le\frac{Nq}{2}\|f\|_{\ast,B}.
\end{equation}
Both constants are sharp. At depth $m\ge1$, the second sharp constant is
$\min\{m,N\}q/2$. For $1<p<\infty$ one has the bound
\begin{equation}\label{eq:cumulative-interpolation}
 \mathcal H_{p,N;B}(f)\le
 \left(\frac{Nq}{2}\right)^{1-1/p}\|f\|_{\ast,B}.
\end{equation}
No sharpness assertion for \eqref{eq:cumulative-interpolation} is made when
$N>1$.
\end{thm}
\begin{proof}
On each \(J\in\cT_N(I)\), Jensen's inequality gives
\(|f_J-f_I|\le\avg{|f-f_I|}_J\). Averaging these
inequalities over \(I\) yields $ \avg{|E_N^I f-f_I|}_I\le\osc_I(f)$. Along each path of $N$ edges,
\cref{lem:edge} bounds the change in conditional average by
$Nq\|f\|_{\ast,B}/2$. This proves \eqref{eq:cumulative-endpoints}.
The inequality
$\avg{|u|^p}\le\|u\|_\infty^{p-1}\avg{|u|}$ gives
\eqref{eq:cumulative-interpolation}. If $f$ has depth $m$, there are at most
$\min\{m,N\}$ nonzero changes on any path.

A centered depth-one function proves sharpness of the first endpoint.
For the other endpoint put $h=q/2$ and select a chain
$I_0\supset I_1\supset\cdots\supset I_N$, each node a child of its
predecessor. Prescribe $(f)_{I_k}=kh$. At stage $k$, make the other
$q-1$ children of $I_{k-1}$ constant with value
\[
 y_k=(k-1)h-\frac{h}{q-1},
\]
and put $f=Nh$ on $I_N$. The averaging relation is satisfied at every
stage, so $f_{I_0}=0$.

We check the full terminal oscillation, rather than only the changes in
conditional means. On the active child $I_k$, every terminal value is at
least $(k-1)h$: a later exit at stage $\ell\ge k+1$ has value
$(\ell-1)h-h/(q-1)\ge(k-1)h$, and the final value has the same property.
On each sibling of $I_k$, the deviation from $(k-1)h$ is $-h/(q-1)$.
Consequently the positive and negative contributions to the normalized
oscillation of $I_{k-1}$ are each $h/q=1/2$. Thus
$\osc_{I_{k-1}}(f)=1$. All remaining nodes are constant, and
$\|f\|_{\ast,I_0}=1$. At the root \(I_0\), the value of \(E_N^{I_0}f\) on
\(I_N\) is \(Nh\), whereas \(f_{I_0}=0\). Hence
 $\Gamma_{\infty,N}(I_0;f)\ge Nh=Nq/2,$
which matches the upper bound.
For the depth-\(m\) assertion, carry out the construction for
\(k=\min\{m,N\}\) generations. If \(m>k\), give every descendant
of a generation-\(k\) leaf the same value as that leaf. The
resulting depth-\(m\) function still has BMO seminorm \(1\),
and its largest change from the root average after \(N\)
generations is \(kq/2\).
\end{proof}
These envelopes measure changes in averages along tree
paths. The one-generation constants are explicit for every \(p\),
and the cumulative constants are explicit at \(p=1\) and
\(p=\infty\). We now
assign increasing weights to successive generations. This
keeps every finite-depth problem explicit while allowing its
sharp constants to change at every depth.

\subsection{Weighted differences and nonstabilizing finite constants}\label{subsec:the weighted example}
Let $0<w_1\le w_2\le\cdots$ be bounded, and write $W=\sup_rw_r$.
With $E_0^Bf=f_B$, define
\begin{equation}\label{eq:weighted-envelope}
 \mathcal W_B(f):=\sup_{r\ge1}w_r
       \|E_r^Bf-E_{r-1}^Bf\|_{\operatorname{L}^\infty(B)}, \qquad
 \mathcal W(f):=\sup_{B\in\cT}\mathcal W_B(f).
\end{equation}
This functional is homogeneous, translation invariant, and invariant under
rooted tree isomorphisms.
For a depth-\(m\) vector \(a\), the finite-depth restriction is
\[
 \mathcal W_m(a)
 =\max_{1\le r\le m}w_r
   \max_{|\sigma|=r}
   |A_\sigma(a)-A_{\sigma^P}(a)|,
\]
where \(\sigma^P\) is the parent of \(\sigma\).
Thus \(\mathcal W_m\) is continuous and piecewise linear.
It is translation invariant, so subtracting the root average
provides an anchor. If \(w_1,\ldots,w_m\) are rational, then
\(\mathcal W_m\) is rational piecewise linear, and
\cref{thm:finite-anchor} gives a rational finite-depth
extremizer and a rational sharp constant.
Monotonicity of the weights makes it hereditary:
a difference at relative level $r$ below a depth-$d$ descendant has weight
$w_{r+d}\ge w_r$ when viewed from $B$.

\begin{prop}\label{prop:weighted}
The exact finite-depth and local constants of \eqref{eq:weighted-envelope} are
\[
 c_{\mathcal W,m}=\frac q2w_m,
 \qquad c_{\mathcal W}(B)=c_{\mathcal W}=\frac q2W.
\]
Each finite-depth value has a centered extremizer. If every $w_r$ is rational,
the finite functionals are rational piecewise-linear.
\end{prop}
\begin{proof}
The edge estimate gives
$\mathcal W_B(f)\le(q/2)W\|f\|_{\ast,B}$. For a depth-$m$ function, all
levels after $m$ vanish, so the upper bound is $(q/2)w_m$.
To attain it, choose a node at depth $m-1$, put the values
$(q/2,-q/2,0,\ldots,0)$ on its children, and set $f=0$ outside that node.
Its mean is zero on every ancestor, its BMO seminorm is $1$, and its only
nonzero martingale difference is at level $m$, with maximum $q/2$.
Thus $c_{\mathcal W,m}=(q/2)w_m$.

For arbitrary $f$, the first $m$ differences of $E_m^Bf$ equal those of $f$,
and all later differences vanish. Hence
$\mathcal W_B(E_m^Bf)\longrightarrow\mathcal W_B(f)$.
The reduction theorem now identifies the local and global constants with
$\sup_m(q/2)w_m$. The finite formula is a maximum of absolute values of
linear forms with rational coefficients whenever $w_1,\ldots,w_m$ are
rational.
\end{proof}
For $w_m=1-2^{-m}$, every finite-depth constant is rational and
$c_{\mathcal W,m}<q/2=c_{\mathcal W}$ for every $m$. For
 $w_m=2^{-m}\lfloor2^m\sqrt2\rfloor$,
the weights are positive rational numbers, are nondecreasing, and converge
to $\sqrt2$. Thus every finite-depth constant is rational, whereas
$c_{\mathcal W}=q\sqrt2/2$ is irrational.

These examples concern the finite-depth constants at a
fixed root \(B\). For the global envelope, choose a centered
one-generation extremizer on a node \(I\) and extend it by
zero outside \(I\). Viewed from ancestors for which \(I\)
lies successively deeper, its nonzero difference receives
the weights \(w_1,w_2,\ldots\). The global envelope of this
single function is therefore \(qW/2\). Thus failure of
finite-depth stabilization at a fixed root does not prevent
attainment of the global sharp constant on the bi-infinite tree.

\section{Further questions}\label{sec:questions}
Our formalism identifies the local and global sharp
constants with the supremum of the finite-depth constants. It does not determine whether a
particular sequence of finite-depth constants stabilizes. By \cref{thm:stable-centering}, the centered and uncentered
finite-depth problems have the same supremum when taken over
all depths. For absolute value,
the exact value of $c(q)$ remains the principal question. The bounds
proved here are
\[
 \frac{25}{24}\le c(2)\le\frac{18}{11},\qquad
 \Delta_q+\frac{q-2}{q^3}\le c(q)\le
       \frac{2(q+1)^2}{q^2+3q+1}\quad(q\ge3).
\]

There are several concrete remaining problems.
\begin{enumerate}
\item Determine \(c(q)\), beginning with \(q=2\).
In particular, find a sharp upper bound that accounts for
how the distributions on a node and its descendants constrain
one another. Applying the scalar estimate of
\cref{lem:folding-tail} independently at each node does not
capture these multigeneration constraints.

\item Find the least depth at which absolute value increases
the BMO seminorm on a binary tree, and determine further
exact finite-depth constants. We know
\(c_1(2)=c_2(2)=1\) and \(c_4(2)\ge41/40\), but have not
determined \(c_3(2)\) or \(c_4(2)\). The reduction to centered cases in
\Cref{prop:absolute-structure} and the family of finite linear programs \eqref{eq:abs-LP} provide a precise finite certification problem. The linear programs in \eqref{eq:abs-LP} compute the
uncentered constants. Analogous programs compute the
centered constants by imposing zero root average and
allowing the output oscillation to be evaluated at any node.

\item Decide whether \(c_m(q)=c(q)\) for some finite \(m\).
If equality holds, rationality and attainment at depth \(m\) (\cref{prop:absolute-structure})
give a rational sharp constant and a finite-depth
extremizer. Stable centering (\cref{thm:stable-centering}) then gives a centered
extremizer at a possibly greater finite depth. Whether rationality of \(c(q)\) would imply stabilization
for the absolute-value family is unresolved. For general
functionals, rationality alone is insufficient: the
weighted example in \Cref{subsec:the weighted example} with \(w_m=1-2^{-m}\) has rational global
constant and no finite-depth stabilization.
\item Determine optimal centering bounds for non-homogenueous trees in which every node has at least two children and each child has relative measure at least \(\alpha>0\). The construction of \cref{thm:stable-centering} gives a
finite depth bound depending on \(\alpha\) and the initial root mean; the question is the best possible dependence.
Norm-preserving extension (\cref{lem:extension}) and finite-depth approximation (\cref{lem:expectation})
continue to apply under the same differentiation hypothesis.
The question is how many ancestor generations are needed
to center an admissible configuration while keeping its BMO
seminorm at most \(1\). If finite-depth constants are to be
independent of the root, the tree isomorphisms used in
\cref{def:scN} must preserve relative child measures.
Rational child proportions would retain rational linear
constraints for the finite \(\operatorname{L}^1\) problems.
\end{enumerate}
The examples reveal different extremal behaviour within the same
tree geometry. An optimal translation may be determined by a node
average or by a leaf value, while finite-depth constants can behave
differently as the depth grows. The formalism treats these features
separately and provides a setting for studying other functionals
arising in the analysis of tree BMO.

\end{document}